\documentclass[preprint,3p,11pt,authoryear]{elsarticle}

\makeatletter
\def\ps@pprintTitle{%
 \let\@oddhead\@empty
 \let\@evenhead\@empty
 \def\@oddfoot{}%
 \let\@evenfoot\@oddfoot}
\makeatother

\usepackage{amssymb,amsthm,mathtools}
\usepackage{lineno}  

\usepackage{enumerate,natbib,geometry}
\usepackage{appendix} 
\usepackage{mathrsfs} 
\usepackage{dsfont} 

\usepackage[colorlinks]{hyperref}

\newtheorem{theorem}{Theorem}[section]

\newtheorem{corollary}[theorem]{Corollary}
\newtheorem{lemma}[theorem]{Lemma}
\newtheorem*{remark}{Remark}

\makeatletter \@addtoreset{equation}{section} \makeatother

\newcommand{\N}{\mathbb{N}}
\newcommand{\R}{\mathbb{R}}

\newcommand{\QQ}{\mathbb{Q}}
\newcommand{\PP}{\mathbb{P}}
\newcommand{\EE}{\mathbb{E}}

\newcommand{\VV}{\mathbb{V}\mathrm{ar}}

\newcommand{\bb}[1]{\boldsymbol{#1}}

\newcommand{\OO}{\mathcal O}
\newcommand{\oo}{\mathrm{o}}
\newcommand{\leqdef}{\vcentcolon=}
\newcommand{\reqdef}{=\vcentcolon}
\newcommand{\rd}{{\rm d}}
\newcommand{\ind}{\mathds{1}}

\allowdisplaybreaks

\begin{document}

\begin{frontmatter}

    \title{A precise local limit theorem for the multinomial distribution\\and some applications}%

    \author[a1]{Fr\'ed\'eric Ouimet\texorpdfstring{\corref{cor1}\fnref{fn1}}{)}}%

    \address[a1]{California Institute of Technology, Pasadena, USA.}%

    \cortext[cor1]{Corresponding author}%
    \ead{ouimetfr@caltech.edu}%


    \begin{abstract}
        In \cite{MR750392}, a precise local limit theorem for the multinomial distribution is derived by inverting the Fourier transform, where the error terms are explicit up to order $N^{-1}$.
        In this paper, we give an alternative (conceptually simpler) proof based on Stirling's formula and a careful handling of Taylor expansions, and we show how the result can be used to approximate multinomial probabilities on most subsets of $\R^d$.
        Furthermore, we discuss a recent application of the result to obtain asymptotic properties of Bernstein estimators on the simplex, we improve the main result in \cite{MR1922539} on the Le Cam distance bound between multinomial and multivariate normal experiments while simultaneously simplifying the proof, and we mention another potential application related to finely tuned continuity corrections.
    \end{abstract}


\end{frontmatter}

\vspace{-1mm}
\section{Introduction}\label{sec:intro}

Given a set of probability weights $\bb{p}\in (0,1)^d$ that satisfies $\|\bb{p}\|_1 \leqdef \sum_{i=1}^d |p_i| < 1$, the $\mathrm{Multinomial}\hspace{0.2mm}(N,\bb{p})$ probability mass function is defined by
\vspace{-1mm}
\begin{equation}\label{eq:multinomial.pdf}
    p_N(\bb{k}) = \frac{N!}{(N - \|\bb{k}\|_1)! \prod_{i=1}^d k_i!} \cdot q^{N - \|k\|_1} \prod_{i=1}^d p_i^{k_i}, \quad \bb{k}\in \N_0^d, ~ \|\bb{k}\|_1 \leq N,
\end{equation}
where $q \leqdef 1 - \|\bb{p}\|_1 > 0$ and $N\in \N$.
The covariance matrix of the multinomial distribution is well-known to be $N \, \Sigma$, where $\Sigma \leqdef \text{diag}(\bb{p}) - \bb{p} \bb{p}^{\top}$,
see e.g.\ \cite[p.377]{MR2168237}.
From Theorem 1 in \cite{MR1157720}, we also know that $\det(\Sigma) = p_1p_2 \dots p_d \hspace{0.3mm} q$.
The purpose of this paper is to establish an asymptotic expansion for \eqref{eq:multinomial.pdf} in terms of the multivariate normal density with the same covariance profile, namely:
\begin{equation}\label{eq:phi.M}
    \phi_{\Sigma}(\bb{x}) \leqdef \frac{1}{\sqrt{(2\pi)^d \, p_1 p_2 \dots p_d \hspace{0.3mm} q}} \cdot \exp\Big(-\frac{1}{2} \bb{x}^{\top} \Sigma^{-1} \, \bb{x}\Big), \quad \bb{x}\in \R^d.
\end{equation}

\vspace{1mm}
\noindent
This kind of expansion can be useful in all sorts of estimation problems; we give three examples in Section~\ref{sec:applications}.
For a general presentation on local limit theorems, see e.g.\ \cite{MR1295242}.

\begin{remark}
    Throughout the paper, the notation $u = \OO(v)$ means that $\limsup_{N\to \infty} |u / v| < C$, where $C\in (0,\infty)$ is a universal constant.
    Whenever $C$ might depend on a parameter, we add a subscript (for example, $u = \OO_d(v)$).
    Similarly, $u = \oo(v)$ means that $\lim_{N\to \infty} |u / v| = 0$, and subscripts indicate which parameters the convergence rate can depend on.
\end{remark}

\section{Main result}\label{sec:main.result}

General local asymptotic expansions of probabilities related to the sums of lattice random vectors are well-known in the literature, see e.g.\ Theorem 1 in \cite{doi:10.1137/1114060}, Theorem 1 in \cite{doi:10.1007/BF00967926}, Theorem 22.1 in \cite{MR0436272}, etc.
However, the error terms in these expansions must be estimated themselves and as such are not explicit enough for applications.
By using the specificity of the distribution at hand, it is often possible to refine those results and obtain explicit and exact rates of convergence with a fraction of the mathematical machinery.

\vspace{3mm}
In the specific case of the multinomial distribution, a local limit theorem (up to an $\OO(N^{-1})$ error in \eqref{eq:LLT.order.2}) was proved for the binomial distribution  on page 141 of \cite{MR56861} and for the multinomial distribution in Lemma 2 of \cite{MR0478288}.
The latter result was extended to a version of \eqref{eq:LLT.order.2} that is symmetrized for the $d+1$ variables $\delta_{1,k_1},\dots,\delta_{d,k_d},-\sum_{i=1}^d \delta_{i,k_i}$ in \cite{MR750392}  by inverting the Fourier transform.%
\footnote{This was pointed out by a referee and was unknown to us at the time of writing the first draft.}
%


\vspace{3mm}
In this paper, we offer an alternative proof that we believe is conceptually simpler.
It is based on Stirling's formula and a careful handling of several Taylor expansions.
The computations generalize the ones on pages 437-438 of \cite{MR538319}, which were used to obtain a finely tuned continuity correction for the survival function of the binomial distribution (see the potential application in Section~\ref{sec:other.applications}).

\begin{theorem}[Local limit theorem]\label{thm:p.k.expansion}
    Pick any $\eta\in (0,1)$ and let
    \begin{equation}\label{eq:thm:p.k.expansion.condition}
        B_{N\hspace{-0.5mm},\bb{p}}(\eta) \leqdef \left\{\bb{k}\in \N_0^d : \bigg|\frac{\delta_{i,k_i}}{\sqrt{N} p_i}\bigg| \leq \eta N^{-1/3}, ~\text{for all } i\in \{1,2,\dots,d\}, ~\text{and} ~ \bigg|\sum_{i=1}^d \frac{\delta_{i,k_i}}{\sqrt{N} q}\bigg| \leq \eta N^{-1/3}\right\}
    \end{equation}
    denote the bulk of the multinomial distribution.
    Then, uniformly for $\bb{k}\in B_{N\hspace{-0.5mm},\bb{p}}(\eta)$, we have
    \begin{equation}\label{eq:LLT.order.2}
        \begin{aligned}
            \frac{p_N(\bb{k})}{N^{-d/2} \phi_{\Sigma}(\bb{\delta}_{\bb{k}})}
            &= 1 + N^{-1/2} \cdot \left\{\hspace{-1mm}
                \begin{array}{l}
                    - \frac{1}{2} \sum_{i=1}^d \delta_{i,k_i} \big\{\frac{1}{p_i} - \frac{1}{q}\big\} \\[2mm]
                    + \frac{1}{6} \sum_{i\hspace{-0.1mm},\hspace{0.1mm} j\hspace{-0.1mm},\hspace{0.1mm} \ell = 1}^d \delta_{i,k_i} \delta_{j,k_j} \delta_{\ell,k_\ell} \big\{\frac{1}{p_i^2} \bb{1}_{\{i = j = \ell\}} - \frac{1}{q^2}\big\}
                \end{array}
                \hspace{-1mm}\right\} \\[2mm]
            &\qquad+ N^{-1} \cdot \left\{\hspace{-1mm}
                \begin{array}{l}
                    - \frac{1}{12} \sum_{i\hspace{-0.1mm},\hspace{0.1mm} j\hspace{-0.1mm},\hspace{0.1mm} \ell \hspace{-0.1mm},\hspace{-0.1mm} m = 1}^d \delta_{i,k_i} \delta_{j,k_j} \delta_{\ell,k_\ell} \delta_{m,k_m} \big\{\frac{1}{p_i^2} \bb{1}_{\{i = j = \ell\}} - \frac{1}{q^2}\big\} \big\{\frac{1}{p_m} - \frac{1}{q}\big\} \\[2mm]
                    - \frac{1}{12} \sum_{i\hspace{-0.1mm},\hspace{0.1mm} j\hspace{-0.1mm},\hspace{0.1mm} \ell \hspace{-0.1mm},\hspace{-0.1mm} m = 1}^d \delta_{i,k_i} \delta_{j,k_j} \delta_{\ell,k_\ell} \delta_{m,k_m} \big\{\frac{1}{p_i^3} \bb{1}_{\{i = j = \ell = m\}} + \frac{1}{q^3}\big\} \\[2mm]
                    + \frac{1}{72} \big(\sum_{i\hspace{-0.1mm},\hspace{0.1mm} j\hspace{-0.1mm},\hspace{0.1mm} \ell = 1}^d \delta_{i,k_i} \delta_{j,k_j} \delta_{\ell,k_\ell} \big\{\frac{1}{p_i^2} \bb{1}_{\{i = j = \ell\}} - \frac{1}{q^2}\big\}\big)^2 \\[2mm]
                    + \frac{1}{8} \sum_{i \hspace{-0.1mm},\hspace{-0.1mm} j = 1}^d \delta_{i,k_i} \delta_{j,k_j} \big\{\frac{3}{p_i^2} \bb{1}_{\{i = j\}} + \frac{2}{p_i p_j} \bb{1}_{\{i < j\}} - \frac{2}{p_i q} + \frac{3}{q^2}\big\} \\[1.5mm]
                    + \frac{1}{12} \big\{1 - \sum_{i=1}^d p_i^{-1} - q^{-1}\big\}
                \end{array}
                \hspace{-1mm}\right\} \\[2mm]
            &\qquad+ \OO_{d,\bb{p},\eta}\Big(\frac{(1 + \|\bb{\delta}_{\bb{k}}\|_1)^9}{N^{3/2}}\Big), \qquad \text{as $N\to \infty$,}
        \end{aligned}
    \end{equation}
    where $\bb{\delta}_{\bb{b}} \leqdef (\delta_{1,b_1},\delta_{2,b_2},\dots,\delta_{d,b_d})^{\top}$ \hspace{-1mm}and
    \begin{equation}\label{def:delta}
        \delta_{i,b} \leqdef \frac{b - N p_i}{\sqrt{N}}, \quad b\in \R, ~i\in \{1,2,\dots,d\}.
    \end{equation}
    It is straightforward to verify that \eqref{eq:LLT.order.2} is equivalent to the following symmetrized version, which can also be found in Lemma 2.1 of \cite{MR750392}:
    \begin{equation}\label{eq:LLT.order.2.symmetrized}
        \begin{aligned}
            \frac{p_N(\bb{k})}{N^{-d/2} \phi_{\Sigma}(\bb{\delta}_{\bb{k}})}
            &= 1 + N^{-1/2} \cdot \left\{\hspace{-1mm}
                \begin{array}{l}
                    - \frac{1}{2} \sum_{i=1}^{d+1} \frac{\delta_{i,k_i}}{p_i} + \frac{1}{6} \sum_{i=1}^{d+1} \delta_{i,k_i} \Big(\frac{\delta_{i,k_i}}{p_i}\Big)^2
                \end{array}
                \hspace{-1mm}\right\} \\[0mm]
            &\qquad+ N^{-1} \cdot \left\{\hspace{-1mm}
                \begin{array}{l}
                    \frac{1}{2} \left\{\hspace{-1mm}
                    \begin{array}{l}
                        - \frac{1}{2} \sum_{i=1}^{d+1} \frac{\delta_{i,k_i}}{p_i} + \frac{1}{6} \sum_{i=1}^{d+1} \delta_{i,k_i} \Big(\frac{\delta_{i,k_i}}{p_i}\Big)^2
                    \end{array}
                    \hspace{-1mm}\right\}^2 \\[2mm]
                    + \frac{1}{4} \sum_{i=1}^{d+1} \Big(\frac{\delta_{i,k_i}}{p_i}\Big)^2
                    - \frac{1}{12} \sum_{i=1}^{d+1} \delta_{i,k_i} \Big(\frac{\delta_{i,k_i}}{p_i}\Big)^3 \\[2mm]
                    + \frac{1}{12} \big\{1 - \sum_{i=1}^{d+1} p_i^{-1}\big\}
                \end{array}
                \hspace{-1mm}\right\} \\[1mm]
            &\qquad+ \OO_{d,\bb{p},\eta}\Big(\frac{(1 + \|\bb{\delta}_{\bb{k}}\|_1)^9}{N^{3/2}}\Big), \qquad \text{as $N\to \infty$,}
        \end{aligned}
    \end{equation}
    where $\delta_{d+1,k_{d+1}} \leqdef -\sum_{i=1}^d \delta_{i,k_i}$ and $p_{d+1} \leqdef q = 1 - \|\bb{p}\|_1$.
\end{theorem}

With the expansion in \eqref{eq:LLT.order.2}, we can easily approximate multinomial probabilities on any subset $A\subseteq \N_0^d$ using Riemann integrals.
This is an advantage of the representation \eqref{eq:LLT.order.2} over the symmetrized version \eqref{eq:LLT.order.2.symmetrized}.

\begin{corollary}\label{cor:approx.proba}
    For any set $A\subseteq \N_0^d$, let
    \begin{equation}
        \mathcal{H}_A \leqdef \bigcup_{\substack{\bb{k}\in A \\ \|\bb{k}\|_1 \leq N}} \, \Big\{[\delta_{1,k_1 - \frac{1}{2}}, \delta_{1,k_1 + \frac{1}{2}}] \times \dots \times [\delta_{d,k_d - \frac{1}{2}}, \delta_{d,k_d + \frac{1}{2}}]\Big\}
    \end{equation}
    denote the union of the normalized unit hypercubes centered at $\bb{k} - N \bb{p}$ for all points $\bb{k}$ that are both in $A$ and in the $d$-dimensional simplex of width $N$.
    Then,
    \begin{equation}\label{eq:cor:approx.proba}
        \begin{aligned}
            &\sum_{\substack{\bb{k}\in A \\ \|\bb{k}\|_1 \leq N}} \hspace{-1mm} p_N(\bb{k}) = \int_{\mathcal{H}_A} \hspace{-1mm} \phi_{\Sigma}(\bb{y}) \rd \bb{y} + N^{-1/2} \cdot \left\{\hspace{-1mm}
                \begin{array}{l}
                    - \frac{1}{2} \sum_{i=1}^d \big\{\frac{1}{p_i} - \frac{1}{q}\big\} \int_{\mathcal{H}_A} y_i \, \phi_{\Sigma}(\bb{y}) \rd \bb{y} \\[2mm]
                    + \frac{1}{6} \sum_{i\hspace{-0.1mm},\hspace{0.1mm} j\hspace{-0.1mm},\hspace{0.1mm} \ell = 1}^d \big\{\frac{1}{p_i^2} \bb{1}_{\{i = j = \ell\}} - \frac{1}{q^2}\big\} \int_{\mathcal{H}_A} y_i \, y_j \, y_{\ell} \, \phi_{\Sigma}(\bb{y}) \rd \bb{y}
                \end{array}
                \hspace{-1mm}\right\} \\[0mm]
            &\qquad+ N^{-1} \cdot \left\{\hspace{-1mm}
                \begin{array}{l}
                    -\frac{1}{24} \sum_{i=1}^d \int_{\mathcal{H}_A} \big\{\big([\Sigma^{-1} \bb{\delta}_{\bb{k}}]_i\big)^2 - [\Sigma^{-1}]_{ii}\big\} \, \phi_{\Sigma}(\bb{y}) \rd \bb{y} \\[2mm]
                    - \frac{1}{12} \sum_{i\hspace{-0.1mm},\hspace{0.1mm} j\hspace{-0.1mm},\hspace{0.1mm} \ell \hspace{-0.1mm},\hspace{-0.1mm} m = 1}^d \big\{\frac{1}{p_i^2} \bb{1}_{\{i = j = \ell\}} - \frac{1}{q^2}\big\} \big\{\frac{1}{p_m} - \frac{1}{q}\big\} \int_{\mathcal{H}_A} y_i \, y_j \, y_{\ell} \, y_m \, \phi_{\Sigma}(\bb{y}) \rd \bb{y} \\[2mm]
                    - \frac{1}{12} \sum_{i\hspace{-0.1mm},\hspace{0.1mm} j\hspace{-0.1mm},\hspace{0.1mm} \ell \hspace{-0.1mm},\hspace{-0.1mm} m = 1}^d \big\{\frac{1}{p_i^3} \bb{1}_{\{i = j = \ell = m\}} + \frac{1}{q^3}\big\} \int_{\mathcal{H}_A} y_i \, y_j \, y_{\ell} \, y_m \, \phi_{\Sigma}(\bb{y}) \rd \bb{y} \\[2mm]
                    + \frac{1}{72} \sum_{i\hspace{-0.1mm},\hspace{0.1mm} j\hspace{-0.1mm},\hspace{0.1mm} \ell \hspace{-0.1mm},\hspace{-0.1mm} i'\hspace{-0.5mm},\hspace{0.3mm} j'\hspace{-0.5mm},\hspace{0.3mm} \ell' = 1}^d
                    \Bigg\{\hspace{-1.5mm}
                    \begin{array}{l}
                        \big\{\frac{1}{p_i^2} \bb{1}_{\{i = j = \ell\}} - \frac{1}{q^2}\big\} \\[2mm]
                        \cdot \, \big\{\frac{1}{p_{i'}^2} \bb{1}_{\{i' = j' = \ell'\}} - \frac{1}{q^2}\big\}
                    \end{array}
                    \hspace{-1.5mm}\Bigg\}
                    \int_{\mathcal{H}_A} y_i \, y_j \, y_{\ell} \, y_{i'} \, y_{j'} \, y_{\ell'} \, \phi_{\Sigma}(\bb{y}) \rd \bb{y} \\[2mm]
                    + \frac{1}{8} \sum_{i \hspace{-0.1mm},\hspace{-0.1mm} j = 1}^d \big\{\frac{3}{p_i^2} \bb{1}_{\{i = j\}} + \frac{2}{p_i p_j} \bb{1}_{\{i < j\}} - \frac{2}{p_i q} + \frac{3}{q^2}\big\} \int_{\mathcal{H}_A} y_i \, y_j \, \phi_{\Sigma}(\bb{y}) \rd \bb{y} \\[1.5mm]
                    + \frac{1}{12} \big\{1 - \sum_{i=1}^d p_i^{-1} - q^{-1}\big\} \int_{\mathcal{H}_A} \phi_{\Sigma}(\bb{y}) \rd \bb{y}
                \end{array}
                \hspace{-1mm}\right\} \\[2mm]
            &\qquad+ \OO_{d,\bb{p}}(N^{-3/2}), \qquad \text{as } N\to \infty.
        \end{aligned}
    \end{equation}
    In particular, for any set $\widetilde{A}\subseteq \R^d$ such that
    \begin{equation}
        \int_{\frac{\widetilde{A} - N \bb{p}}{\sqrt{N}} \backslash \mathcal{H}_{\widetilde{A} \cap \N_0^d}} \phi_{\Sigma}(\bb{y}) \rd \bb{y} = \OO_{d,\bb{p}}(N^{-1/2}),
    \end{equation}
    (i.e.\ the boundary is not dominant) we have
    \begin{equation}
        \sum_{\substack{\bb{k}\in \widetilde{A} \\ \|\bb{k}\|_1 \leq N}} \hspace{-1mm} p_N(\bb{k}) = \int_{\frac{\widetilde{A} - N \bb{p}}{\sqrt{N}}} \, \phi_{\Sigma}(\bb{y}) \rd \bb{y} + \OO_{d,\bb{p}}(N^{-1/2}).
    \end{equation}
\end{corollary}

\newpage
\section{Applications}\label{sec:applications}

Before turning to the proofs, we present two applications of Theorem~\ref{thm:p.k.expansion} related to asymptotic properties of Bernstein estimators (Section~\ref{sec:Bernstein.estimators}) and the Le Cam distance between multinomial and multivariate normal experiments (Section~\ref{sec:deficiency.distance}).
We also briefly mention another potential application related to finely tuned continuity corrections (Section~\ref{sec:other.applications}).

    \subsection{Asymptotic properties of Bernstein estimators}\label{sec:Bernstein.estimators}

        In \cite{MR0397977}, \cite{MR1910059} and \cite{MR2960952}, various asymptotic properties for Bernstein estimators of density functions and cumulative distribution functions (c.d.f.s) on the compact interval $[0,1]$ were studied, namely: bias, variance, mean squared error, mean integrated squared error, asymptotic normality and uniform strong consistency.
        When the observations are supported on the $d$-dimensional simplex, we can extend the definition of these estimators and study their asymptotic properties using the local limit theorem (Theorem~\ref{thm:p.k.expansion}).
        Precisely, assume that the observations $\bb{X}_1, \bb{X}_2, \dots, \bb{X}_n$ are independent, $F$ distributed (with density $f$) and supported on the simplex
        \begin{equation}\label{eq:def:simplex}
            \mathcal{S}\leqdef \big\{\bb{p}\in [0,1]^d: \|\bb{p}\|_1 \leq 1\big\}.
        \end{equation}
        Then, for $n, N\in \N$, let
        \begin{equation}\label{eq:cdf.Bernstein.estimator}
            F_{n,N}^{\star}(\bb{p}) \leqdef \sum_{\bb{k}\in \N_0^d \cap N \mathcal{S}} \left\{n^{-1} \sum_{i=1}^n \ind_{(-\bb{\infty}, \frac{\bb{k}}{N}]}(\bb{X}_i)\right\} p_N(\bb{k}), \quad \bb{p}\in \mathcal{S},
        \end{equation}
        be the Bernstein c.d.f.\ estimator on the simplex, and let
        \begin{equation}\label{eq:histogram.estimator}
            \hat{f}_{n,N}(\bb{p}) \leqdef \hspace{-3mm}\sum_{\bb{k}\in \N_0^d \cap (N-1) \mathcal{S}} \left\{\frac{(N-1+d)!}{(N-1)!} \cdot \frac{1}{n} \sum_{i=1}^n \ind_{(\frac{\bb{k}}{N}, \frac{\bb{k} + 1}{N}]}(\bb{X}_i)\right\} p_{N-1}(\bb{k}), \quad \bb{p}\in \mathcal{S},
        \end{equation}
        be the Bernstein density estimator on the simplex.
        Assuming that $F$ and $f$ are respectively three-times and two-times continuously differentiable, straightforward calculations (using the independence of the observations, see Sections 6 and 7 in \cite{arXiv:2002.07758} for details) show that
        \begin{align}\label{eq:cdf.Bernstein.estimator.var.asymp}
            \VV(F_{n,N}^{\star}(\bb{p}))
            &= n^{-1} \hspace{0.2mm} \bigg\{\sum_{\bb{k},\bb{\ell}\in \N_0^d \cap N\mathcal{S}} \hspace{-3mm} F((\bb{k} \wedge \bb{\ell}) / N) \, p_N(\bb{k}) p_N(\bb{\ell}) - \Big(\sum_{\bb{k}\in \N_0^d \cap N \mathcal{S}} \hspace{-3mm} F(\bb{k} / N) \, p_N(\bb{k})\Big)^2\bigg\} \notag \\
            &= n^{-1} \cdot \left\{\hspace{-1mm}
                \begin{array}{l}
                    F(\bb{p}) (1 - F(\bb{p})) + \OO_d(N^{-1}) \\[1.5mm]
                    + \sum_{i=1}^d \frac{\partial}{\partial x_i} F(\bb{p}) \sum_{\bb{k},\bb{\ell}\in \N_0^d \cap N\mathcal{S}} ((k_i \wedge \ell_i) / N - x_i) \, p_N(\bb{k}) p_N(\bb{\ell}) \\[1mm]
                    + \sum_{i,j=1}^d \OO\Big(\sum_{\bb{k},\bb{\ell}\in \N_0^d \cap N\mathcal{S}} |k_i / N - x_i| |k_j / N - x_j| \, p_N(\bb{k}) p_N(\bb{\ell})\Big)
                \end{array}
                \hspace{-1mm}\right\},
        \end{align}

        \vspace{-5mm}
        \noindent
        (here $\bb{k} \wedge \bb{\ell} \leqdef (k_i \wedge \ell_i)_{i=1}^d$) and
        \begin{align}\label{eq:histogram.estimator.var.asymp}
            &\VV(\hat{f}_{n,N}(\bb{p})) \notag \\[-1mm]
            &\quad= \frac{1}{n} \, \bigg(\frac{(N - 1 + d)!}{(N - 1)!}\bigg)^2 \left\{\hspace{-1mm}
                \begin{array}{l}
                    \sum_{\bb{k}\in \N_0^d \cap (N-1)\mathcal{S}} \int_{(\frac{\bb{k}}{N}, \frac{\bb{k} + 1}{N}]} \hspace{-0.5mm} f(\bb{y}) \rd \bb{y} \, p_{N-1}^2(\bb{k}) \\
                    -\Big(\sum_{\bb{k}\in \N_0^d \cap (N-1)\mathcal{S}} \int_{(\frac{\bb{k}}{N}, \frac{\bb{k} + 1}{N}]} \hspace{-0.5mm} f(\bb{y}) \rd \bb{y} \, p_{N-1}(\bb{k})\Big)^2
                \end{array}
                \hspace{-1mm}\right\} \notag \\
            &\quad= n^{-1} N^{d/2} \left\{\hspace{-1mm}
                \begin{array}{l}
                    (f(\bb{p}) + \OO_d(N^{-1})) \Big[(N - 1)^{d/2} \sum_{\bb{k}\in \N_0^d \cap (N-1)\mathcal{S}} p_{N-1}^2(\bb{k})\Big] \\[2.5mm]
                    + \sum_{i=1}^d \OO_d\left(\hspace{-1mm}
                        \begin{array}{l}
                            \sqrt{\sum_{\bb{k}\in \N_0^d \cap (N-1)\mathcal{S}} |k_i / N - x_i|^2 \, p_{N-1}(\bb{k})} \\
                            \cdot \, \sqrt{(N - 1)^d \sum_{\bb{k}\in \N_0^d \cap (N-1)\mathcal{S}} p_{N-1}^3(\bb{k})}
                        \end{array}
                        \hspace{-1mm}\right)
                    + \OO(N^{-d/2})
                \end{array}
                \hspace{-1mm}\right\}.
        \end{align}

        In \cite{arXiv:2002.07758}, the local limit theorem (Theorem~\ref{thm:p.k.expansion}) was applied to show that, for all $\bb{p}\in (0,1)^d$ such that $\|\bb{p}\|_1 < 1$, we have, as $N\to \infty$,
        \begin{align}
            (N-1)^{d/2} \hspace{-1mm} \sum_{\bb{k}\in \N_0^d \cap (N-1)\mathcal{S}} p_{N-1}^2(\bb{k})
            &= \int_{\R^d} \phi_{\Sigma}^2(\bb{y}) \rd \bb{y} + \oo_{d,\bb{p}}(1) \notag \\[-3mm]
            &= \frac{2^{-d/2}}{\sqrt{(2\pi)^d \det(\Sigma)}} \int_{\R^d} \phi_{\frac{1}{2} \Sigma}(\bb{y}) \rd \bb{y} + \oo_{d,\bb{p}}(1) \notag \\
            &= \frac{2^{-d/2}}{\sqrt{(2\pi)^d \det(\Sigma)}} \cdot 1 + \oo_{d,\bb{p}}(1) \notag \\[1mm]
            &= \big[(4\pi)^d p_1 p_2 \dots p_d (1 - \|\bb{p}\|_1)\big]^{-1/2} + \oo_{d,\bb{p}}(1), \\[2mm]
            (N-1)^d \hspace{-1mm} \sum_{\bb{k}\in \N_0^d \cap (N-1)\mathcal{S}} p_{N-1}^3(\bb{k})
            &= \int_{\R^d} \phi_{\Sigma}^3(\bb{y}) \rd \bb{y} + \oo_{d,\bb{p}}(1) \notag \\[-3mm]
            &= \frac{3^{-d/2}}{(2\pi)^d \det(\Sigma)} \int_{\R^d} \phi_{\frac{1}{3} \Sigma}(\bb{y}) \rd \bb{y} + \oo_{d,\bb{p}}(1) \notag \\
            &= \frac{3^{-d/2}}{(2\pi)^d \det(\Sigma)} \cdot 1 + \oo_{d,\bb{p}}(1) \notag \\[1.5mm]
            &= \big[(2\sqrt{3} \hspace{0.2mm}\pi)^d p_1 p_2 \dots p_d (1 - \|\bb{p}\|_1)\big]^{-1} + \oo_{d,\bb{p}}(1),
        \end{align}
        and, using integration by parts,
        \begin{align}
            &N^{1/2} \hspace{-1mm} \sum_{\bb{k},\bb{\ell}\in \N_0^d \cap N\mathcal{S}} ((k_i \wedge \ell_i) / N - p_i) \, p_N(\bb{k}) p_N(\bb{\ell}) \notag \\[-2.5mm]
            &\hspace{25mm}= 2 \cdot p_i (1 - p_i) \int_{-\infty}^{\infty} \frac{z}{p_i (1 - p_i)} \, \phi_{p_i(1 - p_i)}(z) \int_z^{\infty} \phi_{p_i(1 - p_i)}(y) \rd y \rd z + \oo_{d,\bb{p}}(1) \notag \\[0.5mm]
            &\hspace{25mm}= 2 \cdot p_i (1 - p_i) \, \Big[0 - \int_{-\infty}^{\infty} \phi_{p_i (1 - p_i)}^2(z) \rd z\Big] + \oo_{d,\bb{p}}(1) \notag \\[0.5mm]
            &\hspace{25mm}= \frac{- 2 p_i (1 - p_i)}{\sqrt{4\pi p_i(1 - p_i)}} \int_{-\infty}^{\infty} \phi_{\frac{1}{2} p_i (1 - p_i)}(z) \rd z + \oo_{d,\bb{p}}(1) \notag \\[0.5mm]
            &\hspace{25mm}= -\sqrt{\frac{p_i (1 - p_i)}{\pi}} + \oo_{d,\bb{p}}(1),
        \end{align}
        for all $i\in \{1,2,\dots,d\}$.
        By applying these estimates in \eqref{eq:cdf.Bernstein.estimator.var.asymp} and \eqref{eq:histogram.estimator.var.asymp}, we obtain the asymptotics of the variance for the Bernstein density and c.d.f.\ estimators in the interior of the simplex $\mathcal{S}$.
        From this, other asymptotic expressions can be (and were) derived such as the mean squared error and the mean integrated squared error.
        We can also optimize the bandwidth parameter $N$ with respect these expressions to implement a plug-in selection method, exactly as we would in the setting of traditional multivariate kernel estimators, see e.g.\ \cite[Section 6.5]{MR3329609} or \cite[Section 3.6]{MR3822372}.

        The asymptotic results in \cite{arXiv:2002.07758} nicely complement the works of \cite{MR2270097,MR3474765}, who considered the case of the $d$-dimensional unit hypercube, and the work of \cite{MR1293514}, who previously found asymptotic expressions for the bias, variance and mean squared error of Bernstein density estimators on the two-dimensional simplex.%
        \footnote{Errors in \cite{MR3474765} and related works were corrected in Appendix~B of \cite{MR4213687}.}
        The boundary properties of the density and c.d.f.\ estimators were also investigated in \cite{MR2925964} ($d = 1$) and \cite{arXiv:2006.11756} ($d \geq 1$).
        The local limit theorem (Theorem~\ref{thm:p.k.expansion}) might be used to prove other asymptotic properties or refine known ones.

    \newpage
    \subsection{Deficiency bounds between multinomial and multivariate normal experiments}\label{sec:deficiency.distance}

        In \cite{MR1922539}, the author finds an upper bound on the Le Cam distance (called $\Delta$-distance in \cite{MR1784901})
        between multinomial and multivariate normal experiments.
        His proof relies on an analogous bound for vectors of independent binomial random variables and an inductive argument that reduces the dimension of the binomials/normals comparison by a factor of $2$ at each step. The inductive part of his proof (which is the most difficult part, see Lemma 3) can be removed completely because Theorem~\ref{thm:p.k.expansion} allows us to bound the total variation between multinomial and multivariate normal distributions directly (by adapting the proof of Lemma 2 in his paper).
        The details are provided in Lemma~\ref{lem:prelim.Carter} and Theorem~\ref{thm:bound.deficiency.distance} below.
        For an excellent and concise review on Le Cam's theory for the comparison of statistical models, we refer the reader to \cite{MR3850766}.

        \vspace{3mm}
        The following result is analogous to Lemma 2 in \cite{MR1922539}.

        \begin{lemma}\label{lem:prelim.Carter}
            Let $\bb{K}\sim \mathrm{Multinomial}\hspace{0.2mm}(N,\bb{p})$ and $\bb{U}\sim \mathrm{Uniform}\hspace{0.2mm}(-\tfrac{1}{2},\tfrac{1}{2})^d$, where $\bb{K}$ and $\bb{U}$ are assumed independent.
            Define $\bb{X} \leqdef \bb{K} + \bb{U}$ and let $\widetilde{\PP}_{N,\bb{p}}$ be the law of $\bb{X}$.
            In particular, if $\PP_{N,\bb{p}}$ is the law of $\bb{K}$, note that
            \begin{equation}
                \widetilde{\PP}_{N,\bb{p}}(B) \leqdef \int_{N \mathcal{S} \cap \N_0^d} \int_{(-\frac{1}{2},\frac{1}{2})^d} \ind_{B}(\bb{k} + \bb{u}) \rd \bb{u} \, \PP_{N,\bb{p}}(\rd \bb{k}), \quad B\in \mathscr{B}(\R^d).
            \end{equation}
            Let $\QQ_{N,\bb{p}}$ be the law of the multivariate normal distribution $\mathrm{Normal}_d(N \bb{p}, N \, \Sigma)$, where recall $\Sigma \leqdef \mathrm{diag}(\bb{p}) - \bb{p} \bb{p}^{\top}$.
            Then, for all $\bb{p}\in (0,1)^d$ that satisfies $\|\bb{p}\|_1 < 1$, we have, as $N\to \infty$,
            \begin{equation}
                \|\widetilde{\PP}_{N,\bb{p}} - \QQ_{N,\bb{p}}\| = \OO\bigg(N^{-1/2} d \cdot \sqrt{\frac{\max\{p_1,\dots,p_d,q\}}{\min\{p_1,\dots,p_d,q\}}}\bigg),
            \end{equation}
            where $\| \cdot \|$ denotes the total variation norm.
        \end{lemma}

        \begin{proof}
            By the comparison of the total variation norm with the Hellinger distance on page 726 of \cite{MR1922539}, we already know that
            \begin{equation}\label{eq:first.bound.total.variation}
                \|\widetilde{\PP}_{N,\bb{p}} - \QQ_{N,\bb{p}}\| \leq \sqrt{2 \, \PP(\bb{X}\in B_{N\hspace{-0.5mm},\bb{p}}^c(1/2)) + \EE\bigg[\log\Big(\frac{\rd \widetilde{\PP}_{N,\bb{p}}}{\rd \QQ_{N,\bb{p}}}(\bb{X})\Big) \, \ind_{\{\bb{X}\in B_{N\hspace{-0.5mm},\bb{p}}(1/2)\}}\bigg]}.
            \end{equation}
            By applying a union bound followed by Bernstein's inequality for the binomial distribution, we get, for $N$ large enough,
            \begin{align}\label{eq:concentration.bound}
                \PP(\bb{X}\in B_{N\hspace{-0.2mm},\bb{p}}^c(1/2))
                &\leq \sum_{i=1}^d \PP\Bigl(|K_i - N p_i|  > \frac{p_i}{2d} N^{2/3} - 1\Bigr) + \PP\Bigl(\big|\|\bb{K}\|_1 - N (1 - q)\big|  > \frac{q}{2} N^{2/3} - d\Bigr) \notag \\
                &\leq 2 (d+1) \exp\biggl(- \frac{(\min\{p_1,\dots,p_d,q\} N^{2/3} / (2d) - d)^2}{4 N}\biggr) \notag \\
                &\leq 100 \, d \exp\biggl(-\frac{\min\{p_1,\dots,p_d,q\}^2}{100 \, d^{\hspace{0.2mm}2}} \cdot N^{1/3}\biggr).
            \end{align}
            For the expectation in \eqref{eq:first.bound.total.variation}, if $p_N(\bb{x})$ denotes the density function associated with $\widetilde{\PP}_{N,\bb{p}}$ (i.e.\ it is equal to $p_N(\bb{k})$ whenever $\bb{k}\in N \mathcal{S} \cap \N_0^d$ is closest to $\bb{x}$), then
            \begin{align}\label{eq:I.plus.II.plus.III}
                \EE\bigg[\log\bigg(\frac{\rd \widetilde{\PP}_{N,\bb{p}}}{\rd \QQ_{N,\bb{p}}}(\bb{X})\bigg) \, \ind_{\{\bb{X}\in B_{N\hspace{-0.5mm},\bb{p}}(1/2)\}}\bigg]
                &=\EE\bigg[\log\bigg(\frac{p_N(\bb{X})}{N^{-d/2} \phi_{\Sigma}(\bb{\delta}_{\bb{X}})}\bigg) \, \ind_{\{\bb{X}\in B_{N\hspace{-0.5mm},\bb{p}}(1/2)\}}\bigg] \notag \\[1mm]
                &= \EE\bigg[\log\bigg(\frac{p_N(\bb{K})}{N^{-d/2} \phi_{\Sigma}(\bb{\delta}_{\bb{K}})}\bigg) \, \ind_{\{\bb{K}\in B_{N\hspace{-0.5mm},\bb{p}}(1/2)\}}\bigg] \notag \\
                &\quad+ \EE\bigg[\log\bigg(\frac{N^{-d/2} \phi_{\Sigma}(\bb{\delta}_{\bb{K}})}{N^{-d/2} \phi_{\Sigma}(\bb{\delta}_{\bb{X}})}\bigg) \, \ind_{\{\bb{K}\in B_{N\hspace{-0.5mm},\bb{p}}(1/2)\}}\bigg] \notag \\[1mm]
                &\quad+ \EE\bigg[\log\bigg(\frac{p_N(\bb{K})}{N^{-d/2} \phi_{\Sigma}(\bb{\delta}_{\bb{X}})}\bigg) \, (\ind_{\{\bb{X}\in B_{N\hspace{-0.5mm},\bb{p}}(1/2)\}} - \ind_{\{\bb{K}\in B_{N\hspace{-0.5mm},\bb{p}}(1/2)\}})\bigg] \notag \\[1mm]
                &\reqdef (\mathrm{I}) + (\mathrm{II}) + (\mathrm{III}).
            \end{align}
            By Theorem~\ref{thm:p.k.expansion}, we have
            \begin{align}\label{eq:estimate.I.begin}
                (\mathrm{I})
                &= N^{-1/2} \cdot \EE\left[\left\{\hspace{-1mm}
                    \begin{array}{l}
                        - \frac{1}{2} \sum_{i=1}^d \delta_{i,K_i} \big\{\frac{1}{p_i} - \frac{1}{q}\big\} \\[2mm]
                        + \frac{1}{6} \sum_{i\hspace{-0.1mm},\hspace{0.1mm} j\hspace{-0.1mm},\hspace{0.1mm} \ell = 1}^d \delta_{i,K_i} \delta_{j,K_j} \delta_{\ell,K_\ell} \big\{\frac{1}{p_i^2} \bb{1}_{\{i = j = \ell\}} - \frac{1}{q^2}\big\}
                    \end{array}
                    \hspace{-1mm}\right\} \, \ind_{\{\bb{K}\in B_{N\hspace{-0.5mm},\bb{p}}(1/2)\}}\right] \notag \\[1mm]
                &\quad+ N^{-1} \cdot \OO\left(\hspace{-1mm}
                    \begin{array}{l}
                        \Big|\sum_{i=1}^{d+1} \frac{\EE[(K_i - N p_i)^2]}{N p_i^2}\Big| + \Big|\sum_{\substack{i,j=1 \\ i \neq j}}^{d+1} \frac{\EE[(K_i - N p_i) (K_j - N p_j)]}{N p_i p_j}\Big| \\[2mm]
                        + \Big|\sum_{i=1}^{d+1} \frac{\EE[(K_i - N p_i)^6]}{N^3 p_i^4}\Big| + \Big|\sum_{\substack{i,j=1 \\ i \neq j}}^{d+1} \frac{\EE[(K_i - N p_i)^3 (K_j - N p_j)^3]}{N^3 p_i^2 p_j^2}\Big| \\[2.5mm]
                        + \Big|\sum_{i=1}^{d+1} \frac{\EE[(K_i - N p_i)^4]}{N^2 p_i^3}\Big| + 1 + \sum_{i=1}^{d+1} p_i^{-1}
                    \end{array}
                    \hspace{-1mm}\right) \notag \\
                &\quad+ \OO_{d,\bb{p}}(N^{-3/2}).
            \end{align}
            The expression inside the big $\OO(\cdot)$ term here is crucial to get the correct bound on the Le Cam distance in Theorem~\ref{thm:bound.deficiency.distance}.
            The error terms in Lemma 2 of \cite{MR0478288} would not be enough for this purpose; it is part of the reason why an expression as precise as the one in Theorem~\ref{thm:p.k.expansion} is necessary for this application.
            By Lemma~\ref{lem:Leblanc.2012.boundary.Lemma.1} and Lemma~\ref{lem:moments.4.to.6}, the big $\OO(\cdot)$ term above is
            \begin{equation}\label{eq:estimate.I.next}
                =\OO\bigg(d^{\hspace{0.2mm}2} + \sum_{i=1}^{d+1} p_i^{-1}\bigg) = \OO\bigg(d^{\hspace{0.2mm}2} \cdot \frac{\max\{p_1,\dots,p_d,q\}}{\min\{p_1,\dots,p_d,q\}}\bigg).
            \end{equation}
            By putting \eqref{eq:estimate.I.next} in \eqref{eq:estimate.I.begin} and using Lemma~\ref{lem:Leblanc.2012.boundary.Lemma.1.with.set.A}, we get
            \begin{align}\label{eq:estimate.I}
                (\mathrm{I})
                &= N^{-1/2} \left\{\hspace{-1mm}
                    \begin{array}{l}
                        \frac{N^{-1/2}}{6} \sum_{i,j,\ell=1}^d \left(\hspace{-1mm}
                        \begin{array}{l}
                            2 p_i p_j p_{\ell} - \ind_{\{i = j\}} p_i p_{\ell} - \ind_{\{j = \ell\}} p_i p_j \\
                            - \ind_{\{i = \ell\}} p_j p_{\ell} + \ind_{\{i = j = \ell\}} p_i
                        \end{array}
                        \hspace{-1mm}\right)\big\{\frac{1}{p_i^2} \bb{1}_{\{i = j = \ell\}} - \frac{1}{q^2}\big\} \\[4mm]
                        + \, \OO_d\Big(\frac{(\PP(\bb{K}\in B_{N\hspace{-0.5mm},\bb{p}}^c(1/2)))^{1/4}}{(\min\{p_1,\dots,p_d,q\})^2}\Big)
                    \end{array}
                    \hspace{-1mm}\right\} \notag \\[1mm]
                &\quad+ \OO\bigg(N^{-1} d^{\hspace{0.2mm}2} \cdot \frac{\max\{p_1,\dots,p_d,q\}}{\min\{p_1,\dots,p_d,q\}}\bigg) + \OO_{d,\bb{p}}(N^{-3/2}) \notag \\[2mm]
                &= \OO_d\bigg(\frac{N^{-1/2} \big(\PP(\bb{K}\in B_{N\hspace{-0.5mm},\bb{p}}^c(1/2))\big)^{1/4}}{(\min\{p_1,\dots,p_d,q\})^2}\bigg) + \OO\bigg(N^{-1} d^{\hspace{0.2mm}2} \cdot \frac{\max\{p_1,\dots,p_d,q\}}{\min\{p_1,\dots,p_d,q\}}\bigg).
            \end{align}

            For the term $(\mathrm{II})$ in \eqref{eq:I.plus.II.plus.III},
            \begin{align}
                &\log\bigg(\frac{N^{-d/2} \phi_{\Sigma}(\bb{\delta}_{\bb{K}})}{N^{-d/2} \phi_{\Sigma}(\bb{\delta}_{\bb{X}})}\bigg) \notag \\[1mm]
                &\quad= \frac{N^{-1}}{2} (\bb{X} - N \bb{p})^{\top} \Sigma^{-1} (\bb{X} - N \bb{p}) - \frac{N^{-1}}{2} (\bb{K} - N \bb{p})^{\top} \Sigma^{-1} (\bb{K} - N \bb{p}) \notag \\[1mm]
                &\quad= \frac{N^{-1}}{2} (\bb{X} - \bb{K})^{\top} \Sigma^{-1} (\bb{X} - \bb{K}) + \frac{N^{-1}}{2} \bigg[\hspace{-1mm}
                    \begin{array}{l}
                        (\bb{X} - \bb{K})^{\top} \Sigma^{-1} (\bb{K} - N \bb{p}) \\[0.5mm]
                        + (\bb{K} - N \bb{p})^{\top} \Sigma^{-1} (\bb{X} - \bb{K})
                    \end{array}
                    \hspace{-1mm}\bigg].
            \end{align}
            With our assumption that $\bb{K}$ and $\bb{X} - \bb{K}\sim \mathrm{Uniform}\hspace{0.2mm}(-\tfrac{1}{2},\tfrac{1}{2})^d$ are independent, we get
            \begin{equation}\label{eq:estimate.II}
                \begin{aligned}
                    (\mathrm{II})
                    &= \frac{N^{-1}}{2} \sum_{i=1}^d \frac{(\Sigma^{-1})_{ii}}{12} + \OO_d\bigg(\frac{N^{-1/2} \big(\PP(\bb{K}\in B_{N\hspace{-0.5mm},\bb{p}}^c(1/2))\big)^{1/2}}{(\min\{p_1,\dots,p_d,q\})^2}\bigg) \\
                    &= \frac{N^{-1}}{2} \sum_{i=1}^d \frac{(p_i^{-1} + q^{-1})}{12} + \OO_d\bigg(\frac{N^{-1/2} \big(\PP(\bb{K}\in B_{N\hspace{-0.5mm},\bb{p}}^c(1/2))\big)^{1/2}}{(\min\{p_1,\dots,p_d,q\})^2}\bigg) \\
                    &= \OO\bigg(N^{-1} d^{\hspace{0.2mm}2} \cdot \frac{\max\{p_1,\dots,p_d,q\}}{\min\{p_1,\dots,p_d,q\}}\bigg) + \OO_d\bigg(\frac{N^{-1/2} \big(\PP(\bb{K}\in B_{N\hspace{-0.5mm},\bb{p}}^c(1/2))\big)^{1/2}}{(\min\{p_1,\dots,p_d,q\})^2}\bigg),
                \end{aligned}
            \end{equation}
            where we used the expression $(\Sigma^{-1})_{ij} = p_i^{-1} \ind_{\{i = j\}} + q^{-1}$ found in \citep[eq.21]{MR1157720}.
            
            For the term $(\mathrm{III})$ in \eqref{eq:I.plus.II.plus.III}, the following very rough bound from \eqref{eq:LLT.order.2.symmetrized} in Theorem~\ref{thm:p.k.expansion},
            \begin{equation}
                \log\bigg(\frac{p_N(\bb{K})}{N^{-d/2} \phi_{\Sigma}(\bb{\delta}_{\bb{K}})}\bigg) \, \ind_{\{\bb{K}\in B_{N\hspace{-0.5mm},\bb{p}}(1/2)\}} = \OO\big(N \, d \, (\min\{p_1,\dots,p_d,q\})^{-2}\big),
            \end{equation}
            yields
            \begin{equation}\label{eq:estimate.III}
                (\mathrm{III}) = \OO\bigg(N \, d \, (\min\{p_1,\dots,p_d,q\})^{-2} \, \EE\Big[|\ind_{\{\bb{X}\in B_{N\hspace{-0.5mm},\bb{p}}(1/2)\}} - \ind_{\{\bb{K}\in B_{N\hspace{-0.5mm},\bb{p}}(1/2)\}}|\Big]\bigg).
            \end{equation}
            Putting \eqref{eq:estimate.I}, \eqref{eq:estimate.II} and \eqref{eq:estimate.III} in \eqref{eq:I.plus.II.plus.III}, together with the exponential bound
            \begin{equation}
                \begin{aligned}
                    &\EE\Big[|\ind_{\{\bb{X}\in B_{N\hspace{-0.5mm},\bb{p}}(1/2)\}} - \ind_{\{\bb{K}\in B_{N\hspace{-0.5mm},\bb{p}}(1/2)\}}|\Big] \\
                    &\quad=\EE\Big[|\ind_{\{\bb{X}\in B_{N\hspace{-0.5mm},\bb{p}}(1/2), \bb{K}\in B_{N\hspace{-0.5mm},\bb{p}}^c(1/2)\}} + \ind_{\{\bb{X}\in B_{N\hspace{-0.5mm},\bb{p}}(1/2), \bb{K}\in B_{N\hspace{-0.5mm},\bb{p}}(1/2)\}} - \ind_{\{\bb{K}\in B_{N\hspace{-0.5mm},\bb{p}}(1/2)\}}|\Big] \\[1.5mm]
                    &\quad=\EE\Big[|\ind_{\{\bb{X}\in B_{N\hspace{-0.5mm},\bb{p}}(1/2), \bb{K}\in B_{N\hspace{-0.5mm},\bb{p}}^c(1/2)\}} - \ind_{\{\bb{X}\in B_{N\hspace{-0.5mm},\bb{p}}^c(1/2), \bb{K}\in B_{N\hspace{-0.5mm},\bb{p}}(1/2)\}}|\Big] \\[2mm]
                    &\quad\leq \PP(\bb{K}\in B_{N\hspace{-0.5mm},\bb{p}}^c(1/2)) + \PP(\bb{X}\in B_{N\hspace{-0.5mm},\bb{p}}^c(1/2)) \\
                    &\quad\leq 2 \cdot 100 \, d \exp\Big(-\frac{N^{1/3}}{100 d^{\hspace{0.2mm}2}}\Big),
                \end{aligned}
            \end{equation}
            yields, as $N\to \infty$,
            \begin{equation}\label{eq:I.plus.II.plus.III.end}
                \begin{aligned}
                    \EE\bigg[\log\bigg(\frac{\rd \widetilde{\PP}_{N,\bb{p}}}{\rd \QQ_{N,\bb{p}}}(\bb{X})\bigg) \, \ind_{\{\bb{X}\in B_{N\hspace{-0.5mm},\bb{p}}(1/2)\}}\bigg]
                    &= (\mathrm{I}) + (\mathrm{II}) + (\mathrm{III}) \\
                    &= \OO\bigg(N^{-1} d^{\hspace{0.2mm}2} \cdot \frac{\max\{p_1,\dots,p_d,q\}}{\min\{p_1,\dots,p_d,q\}}\bigg).
                \end{aligned}
            \end{equation}
            Now, putting \eqref{eq:concentration.bound} and \eqref{eq:I.plus.II.plus.III.end} together in \eqref{eq:first.bound.total.variation} gives the conclusion.
        \end{proof}

        The next result improves the main theorem in \cite{MR1922539} (Theorem 1) by removing a factor $\log d$ in \eqref{eq:thm:bound.deficiency.distance.bound} (denoted by $\log m$ in his article).
        Note that this factor is proportional to the number of steps in the inductive argument in \cite{MR1922539}.
        Given the above details, our proof is drastically simpler because Lemma 1 and the inductive part of the proof (Lemma 3) in \cite{MR1922539} have been removed completely (which is coherent with us being able to remove the $\log d$ factor).

        \begin{theorem}[Bound on the Le Cam distance]\label{thm:bound.deficiency.distance}
            For any given $R > 0$, let
            \begin{equation}
                \Theta_R \leqdef \left\{\bb{p}\in (0,1)^d: \|\bb{p}\|_1 < 1 ~~ \text{and} ~~ \frac{\max\{p_1,\dots,p_d,q\}}{\min\{p_1,\dots,p_d,q\}} \leq R\right\}.
            \end{equation}
            Define the experiments
            \begin{alignat*}{6}
                    &\mathscr{P}
                    &&\leqdef &&~\{\PP_{N,\bb{p}}\}_{\bb{p}\in \Theta_R}, \quad &&\PP_{N,\bb{p}} ~\text{is the measure induced by } \mathrm{Multinomial}\hspace{0.2mm}(N,\bb{p}), \\
                    &\mathscr{Q}\hspace{-0.5mm}
                    &&\leqdef &&~\{\QQ_{N,\bb{p}}\}_{\bb{p}\in \Theta_R}, \quad &&\QQ_{N,\bb{p}} ~\text{is the measure induced by } \mathrm{Normal}_d(N \bb{p}, N \Sigma),
            \end{alignat*}
            where recall $\Sigma \leqdef \mathrm{diag}(\bb{p}) - \bb{p} \bb{p}^{\top}$.
            Then, we have the following bound on the Le Cam distance $\Delta(\mathscr{P},\mathscr{Q})$ between $\mathscr{P}$ and $\mathscr{Q}$,
            \begin{equation}\label{eq:thm:bound.deficiency.distance.bound}
                \Delta(\mathscr{P},\mathscr{Q}) \leqdef \max\{\delta(\mathscr{P},\mathscr{Q}),\delta(\mathscr{Q},\mathscr{P})\} \leq C_R \, \frac{d}{\sqrt{N}},
            \end{equation}
            where $C_R$ is a positive constant that depends only on $R$,
            \begin{equation}\label{eq:def:deficiency.one.sided}
                \begin{aligned}
                    \delta(\mathscr{P},\mathscr{Q})
                    &\leqdef \inf_{T_1} \sup_{\bb{p}\in \Theta_R} \bigg\|\int_{N \mathcal{S} \cap \N_0^d} T_1(\bb{k}, \cdot \, ) \, \PP_{N,\bb{p}}(\rd \bb{k}) - \QQ_{N,\bb{p}}\bigg\|, \\
                    \delta(\mathscr{Q},\mathscr{P})
                    &\leqdef \inf_{T_2} \sup_{\bb{p}\in \Theta_R} \bigg\|\PP_{N,\bb{p}} - \int_{\R^d} T_2(\bb{y}, \cdot \, ) \, \QQ_{N,\bb{p}}(\rd \bb{y})\bigg\|, \\
                \end{aligned}
            \end{equation}
            and the infima are taken, respectively, over all Markov kernels $T_1 : (N \mathcal{S} \cap \N_0^d) \times \mathscr{B}(\R^d) \to [0,1]$ and $T_2 : \R^d \times \mathscr{B}(N \mathcal{S} \cap \N_0^d) \to [0,1]$.
        \end{theorem}

        \begin{proof}
            By Lemma~\ref{lem:prelim.Carter}, we get the desired bound on $\delta(\mathscr{P},\mathscr{Q})$ by choosing the Markov kernel $T_1^{\star}$ that adds $\bb{U}$ to $\bb{K}$, namely
            \begin{equation}
                \begin{aligned}
                    T_1^{\star}(\bb{k},B) \leqdef \int_{(-\frac{1}{2},\frac{1}{2})^d} \ind_{B}(\bb{k} + \bb{u}) \rd \bb{u}, \quad \bb{k}\in N \mathcal{S} \cap \N_0^d, ~B\in \mathscr{B}(\R^d).
                \end{aligned}
            \end{equation}
            To get the bound on $\delta(\mathscr{Q},\mathscr{P})$, it suffices to consider a Markov kernel $T_2^{\star}$ that inverts the effect of $T_1^{\star}$, i.e.\ rounding off every components of $\bb{Y}\sim \mathrm{Normal}_d(N \bb{p}, N \Sigma)$ to the nearest integer.
            Then, as explained in Section 5 of \cite{MR1922539}, we get
            \begin{equation}
                \begin{aligned}
                    \delta(\mathscr{Q},\mathscr{P})
                    &\leq \bigg\|\PP_{N,\bb{p}} - \int_{\R^d} T_2^{\star}(\bb{y}, \cdot \, ) \, \QQ_{N,\bb{p}}(\rd \bb{y})\bigg\| \\
                    &= \bigg\|\int_{\R^d} T_2^{\star}(\bb{y}, \cdot \, ) \int_{N \mathcal{S} \cap \N_0^d} T_1^{\star}(\bb{k}, \rd \bb{y}) \, \PP_{N,\bb{p}}(\rd \bb{k}) - \int_{\R^d} T_2^{\star}(\bb{y}, \cdot \, ) \, \QQ_{N,\bb{p}}(\rd \bb{y})\bigg\| \\
                    &\leq \bigg\|\int_{N \mathcal{S} \cap \N_0^d} T_1^{\star}(\bb{k}, \cdot \, ) \, \PP_{N,\bb{p}}(\rd \bb{k}) - \QQ_{N,\bb{p}}\bigg\|,
                \end{aligned}
            \end{equation}
            and we get the same bound by Lemma~\ref{lem:prelim.Carter}.
        \end{proof}

        \newpage
        If we consider the following multivariate normal experiments with independent components

        \vspace{-3mm}
        \begin{alignat*}{6}
            &\widetilde{\mathscr{Q}}
            &&\leqdef &&~\{\widetilde{\QQ}_{N,\bb{p}}\}_{\bb{p}\in \Theta_R}, \quad &&\widetilde{\QQ}_{N,\bb{p}} ~\text{is the measure induced by } \mathrm{Normal}_d(N \bb{p}, N \mathrm{diag}(\bb{p})), \\
            &\mathscr{Q}^{\star}\hspace{-0.5mm}
            &&\leqdef &&~\{\QQ_{N,\bb{p}}^{\star}\}_{\bb{p}\in \Theta_R}, \quad &&\QQ_{N,\bb{p}}^{\star} ~\text{is the measure induced by } \mathrm{Normal}_d(\sqrt{N \bb{p}}, \mathrm{diag}(1/4,\dots,1/4)),
        \end{alignat*}

        \vspace{2mm}
        \noindent
        then \cite[Section 7]{MR1922539} also showed that
        \begin{equation}\label{eq:LeCam.distance.indep.normals}
            \Delta(\mathscr{Q},\widetilde{\mathscr{Q}}) \leq C_R \, \sqrt{\frac{d}{N}} \qquad \text{and} \qquad \Delta(\widetilde{\mathscr{Q}},\mathscr{Q}^{\star}) \leq C_R \, \frac{d}{\sqrt{N}},
        \end{equation}
        using a variance stabilizing transformation,
        with proper adjustments to the deficiencies in \eqref{eq:def:deficiency.one.sided}.

        \begin{corollary}\label{cor:main.theorem.consequence}
            With the same notation as in Theorem~\ref{thm:bound.deficiency.distance}, we have
            \begin{equation}
                \Delta(\mathscr{P},\widetilde{\mathscr{Q}}) \leq C_R \, \frac{d}{\sqrt{N}} \qquad \text{and} \qquad \Delta(\mathscr{P},\mathscr{Q}^{\star}) \leq C_R \, \frac{d}{\sqrt{N}},
            \end{equation}
            for a positive constant $C_R$ that depends only on $R$.
        \end{corollary}

        \begin{proof}
            This is a direct consequence of Theorem~\ref{thm:bound.deficiency.distance}, Equation \eqref{eq:LeCam.distance.indep.normals} and the triangle inequality for the pseudometric $\Delta(\cdot , \cdot)$.
        \end{proof}

    \subsection{Other potential applications}\label{sec:other.applications}

        As \cite{MR538319} did for the binomial distribution, it should be possible to derive a finely tuned continuity correction for the survival function of the multinomial distribution by using the local limit theorem in Theorem~\ref{thm:p.k.expansion}.
        However, in the multidimensional setting ($d \geq 2$), the added liberty on the choice of the correction in each dimension poses non trivial numerical difficulties that need to be resolved.
        This point is left for future research.

        \vspace{3mm}
        It should be mentioned that local limit theorems such as the one in Theorem~\ref{thm:p.k.expansion} can be used for many other purposes; the three examples above are only pointers for new research.
        For instance, in \cite{MR750392}, the authors originally used their approximation of multinomial probabilities to obtain expansions for the cumulative distribution function of the following three statistics:
        \begin{equation}
            \begin{tabular}{l}
                $\bullet$ ~ Pearson's chi-square statistic, $\sum_{i=1}^{d+1} (K_i - N p_i)^2 / (N p_i)$, \\[1mm]
                $\bullet$ ~ the log-likelihood ratio statistic, $2 \sum_{i=1}^{d+1} K_i \log(K_i / (N p_i))$, \\[1mm]
                $\bullet$ ~ the Freeman-Tukey statistic, $4 \sum_{i=1}^{d+1} (\sqrt{K_i} - \sqrt{N p_i})^2$,
            \end{tabular}
        \end{equation}
        where $\bb{K}\sim \mathrm{Multinomial}\hspace{0.2mm}(N,\bb{p})$, $K_{d+1} \leqdef N - \|\bb{K}\|_1$ and $p_{d+1}\leqdef q = 1 - \|\bb{p}\|_1$.
        Some of these results were extended by \cite{MR752006} for the convergence of the more general power divergence statistic
        \begin{equation}
            T_{\lambda}(\bb{K}) = \frac{2}{\lambda (\lambda + 1)} \sum_{i=1}^{d+1} K_i \left[\Big(\frac{K_i}{N p_i}\Big)^{\lambda} - 1\right], \quad \lambda\in \R,
        \end{equation}
        to the chi-square distribution.
        Some lapses in the expansions of \cite{MR750392} and \cite{MR752006}, regarding the rate of convergence of the chi-square approximation, were pointed out and fixed in \cite{MR2499200} (see also \cite{MR3079145}).

        \vspace{3mm}
        Other applications of local limit theorems abound in the literature.
        As mentioned in \cite{MR3825458,Ouimet2019phd}, the special case of the multinomial distribution is worth investigating because there are instances in practice where the distribution that we would like to estimate lives naturally on the d-dimensional simplex. One example is the Dirichlet distribution, which is the conjugate prior of the multinomial distribution in Bayesian estimation. See for example \cite{Lange_1995} for an application in the context of allele frequency estimation in genetics.

\section{Proofs}\label{sec:proofs}

\begin{proof}[Proof of Theorem~\ref{thm:p.k.expansion}]
    Using Stirling's formula,
    \begin{equation}
        \log k! = \frac{1}{2} \log(2\pi) + (k + \tfrac{1}{2}) \log k - k + \frac{1}{12k} + \OO(k^{-3}),
    \end{equation}
    see e.g.\ \cite[p.257]{MR0167642}, and taking logarithms in \eqref{eq:multinomial.pdf}, we obtain
    \begin{align}
        \log p_N(\bb{k})
        &= \log N! - \sum_{i=1}^d \log k_i! - \log (N - \|\bb{k}\|_1)! + \sum_{i=1}^d k_i \log p_i + (N - \|\bb{k}\|_1) \log q \notag \\
        &= -\frac{d}{2} \log(2\pi) - \frac{d}{2} \log N - \frac{1}{2} \sum_{i=1}^d \log p_i - \frac{1}{2} \log q \notag \\
        &\quad- \sum_{i=1}^d (k_i + \tfrac{1}{2}) \log\Big(\frac{k_i}{N p_i}\Big) - (N - \|\bb{k}\|_1 + \tfrac{1}{2}) \log\Big(\frac{N - \|\bb{k}\|_1}{N q}\Big) \notag \\
        &\quad+ \frac{1}{12N} \bigg\{1 - \sum_{i=1}^d \frac{N}{k_i} \hspace{-0.5mm}-\hspace{0.5mm} \frac{N}{N - \|\bb{k}\|_1}\bigg\} + \OO\bigg(\frac{1}{N^3}\bigg\{1 + \sum_{i=1}^d \frac{N^3}{k_i^3} + \frac{N^3}{(N - \|\bb{k}\|_1)^3}\bigg\}\bigg).
    \end{align}
    After some algebraic manipulations, we get
    \begin{align}\label{eq:log.p.before}
        \log p_N(\bb{k})
        &= -\log \sqrt{(2\pi N)^d \, p_1 p_2 \dots p_d \hspace{0.3mm} q} \notag \\
        &\quad- \sum_{i=1}^d k_i \log\Big(\frac{k_i}{N p_i}\Big) - (N - \|\bb{k}\|_1) \log\Big(\frac{N - \|\bb{k}\|_1}{N q}\Big) \notag \\
        &\quad- \frac{1}{2} \sum_{i=1}^d \log\Big(\frac{k_i}{N p_i}\Big) - \frac{1}{2} \log\Big(\frac{N - \|\bb{k}\|_1}{N q}\Big)\notag \\
        &\quad+ \frac{1}{12N} \bigg\{1 - \sum_{i=1}^d \frac{1}{p_i} \Big(\frac{k_i}{N p_i}\Big)^{-1} \hspace{-0.5mm}-\hspace{0.5mm} \frac{1}{q} \Big(\frac{N - \|\bb{k}\|_1}{N q}\Big)^{-1}\bigg\} \notag \\
        &\quad+ \OO\bigg(\frac{1}{N^3}\bigg\{1 + \sum_{i=1}^d \frac{1}{p_i^3} \Big(\frac{k_i}{N p_i}\Big)^{-3} + \frac{1}{q^3} \Big(\frac{N - \|\bb{k}\|_1}{N q}\Big)^{-3}\bigg\}\bigg).
    \end{align}
    By writing $k_i = N p_i + (k_i - N p_i)$ and $N - \|\bb{k}\|_1 = N q - \sum_{i=1}^d (k_i - N p_i)$ in \eqref{eq:log.p.before}, we deduce
    \begin{align}\label{eq:log.p}
        \log p_N(\bb{k})
        &= -\log \sqrt{(2\pi N)^d \, p_1 p_2 \dots p_d \hspace{0.3mm} q} \notag \\[0mm]
        &\quad- \, \underbrace{\sum_{i=1}^d N p_i \Big(1 + \frac{\delta_{i,k_i}}{\sqrt{N} p_i}\Big) \log\Big(1 + \frac{\delta_{i,k_i}}{\sqrt{N} p_i}\Big)}_{\reqdef \, (A)} \notag \\
        &\quad- \, \underbrace{N q \Big(1 - \sum_{i=1}^d \frac{\delta_{i,k_i}}{\sqrt{N} q}\Big) \log\Big(1 - \sum_{i=1}^d \frac{\delta_{i,k_i}}{\sqrt{N} q}\Big)}_{\reqdef \, (B)} \notag \\
        &\quad- \frac{1}{2} \sum_{i=1}^d \log\Big(1 + \frac{\delta_{i,k_i}}{\sqrt{N} p_i}\Big) - \frac{1}{2} \log\Big(1 - \sum_{i=1}^d \frac{\delta_{i,k_i}}{\sqrt{N} q}\Big) \notag \\[2mm]
        &\quad+ \frac{1}{12N} \bigg\{1 - \sum_{i=1}^d \frac{1}{p_i} \Big(1 + \frac{\delta_{i,k_i}}{\sqrt{N} p_i}\Big)^{-1} \hspace{-0.5mm}-\hspace{0.5mm} \frac{1}{q} \Big(1 - \sum_{i=1}^d \frac{\delta_{i,k_i}}{\sqrt{N} q}\Big)^{-1}\bigg\} + \OO_{d,\bb{p},\eta}(N^{-3}).
    \end{align}
    By applying the following Taylor expansions, valid for $\max\{|x|, |\sum_{i=1}^d x_i|\} < \widetilde{\eta} < 1$,
    \begin{equation}
        \begin{aligned}
            &(1 + x) \log(1 + x) = x + \frac{x^2}{2} - \frac{x^3}{6} + \frac{x^4}{12} + \OO((1 - \widetilde{\eta})^{-4} x^5), \\
            &(1 - \sum_{i=1}^d x_i) \log(1 - \sum_{i=1}^d x_i) \\
            &\qquad= -\sum_{i=1}^d x_i + \sum_{i\hspace{-0.1mm},\hspace{0.1mm} j = 1}^d \frac{x_i x_j}{2} + \hspace{-1mm} \sum_{i\hspace{-0.1mm},\hspace{0.1mm} j\hspace{-0.1mm},\hspace{0.1mm} \ell = 1}^d \hspace{-2mm} \frac{x_i x_j x_{\ell}}{6} + \hspace{-2.5mm} \sum_{i\hspace{-0.1mm},\hspace{0.1mm} j\hspace{-0.1mm},\hspace{0.1mm} \ell \hspace{-0.1mm}, \hspace{-0.1mm} m = 1}^d \hspace{-3mm} \frac{x_i x_j x_{\ell} x_m}{12} + \OO_d((1 - \widetilde{\eta})^{-4} \|\bb{x}\|_1^5),
        \end{aligned}
    \end{equation}
    in \eqref{eq:log.p}, we have, respectively,
    \begin{align}\label{eq:log.p.eq.1}
        \mathrm{(A)}
        &= \sum_{i=1}^d N p_i
            \left\{\hspace{-1mm}
            \begin{array}{l}
                \Big(\frac{\delta_{i,k_i}}{\sqrt{N} p_i}\Big) + \frac{1}{2} \Big(\frac{\delta_{i,k_i}}{\sqrt{N} p_i}\Big)^2 - \frac{1}{6} \Big(\frac{\delta_{i,k_i}}{\sqrt{N} p_i}\Big)^3 \\[2mm]
                + \frac{1}{12} \Big(\frac{\delta_{i,k_i}}{\sqrt{N} p_i}\Big)^4 + \OO_{p_i,\eta}\Big(\big(\frac{\delta_{i,k_i}}{\sqrt{N}}\big)^5\Big)
            \end{array}
            \hspace{-1mm}\right\} \notag \\[1mm]
        &= \sum_{i=1}^d (k_i - N p_i) + \sum_{i=1}^d \frac{1}{2} \delta_{i,k_i}^2 \bigg\{\frac{1}{p_i} - \frac{\delta_{i,k_i}}{3\sqrt{N} p_i^2} + \frac{\delta_{i,k_i}^2}{6 N p_i^3} + \OO_{p_i,\eta}\bigg(\Big(\frac{\delta_{i,k_i}}{\sqrt{N}}\Big)^3\bigg)\bigg\},
    \end{align}
    and
    \begin{align}\label{eq:log.p.eq.2}
        \mathrm{(B)}
        &= N q
            \left\{\hspace{-1mm}
            \begin{array}{l}
                -\sum_{i=1}^d \frac{\delta_{i,k_i}}{\sqrt{N} q} + \frac{1}{2} \sum_{i\hspace{-0.1mm},\hspace{0.1mm} j = 1}^d \frac{\delta_{i,k_i}}{\sqrt{N} q} \frac{\delta_{j,k_j}}{\sqrt{N} q} + \frac{1}{6} \sum_{i\hspace{-0.1mm},\hspace{0.1mm} j\hspace{-0.1mm},\hspace{0.1mm} \ell = 1}^d \frac{\delta_{i,k_i}}{\sqrt{N} q} \frac{\delta_{j,k_j}}{\sqrt{N} q} \frac{\delta_{\ell,k_{\ell}}}{\sqrt{N} q} \\[2mm]
                + \frac{1}{12} \sum_{i\hspace{-0.1mm},\hspace{0.1mm} j\hspace{-0.1mm},\hspace{0.1mm} \ell \hspace{-0.1mm}, m = 1}^d \frac{\delta_{i,k_i}}{\sqrt{N} q} \frac{\delta_{j,k_j}}{\sqrt{N} q} \frac{\delta_{\ell,k_{\ell}}}{\sqrt{N} q} \frac{\delta_{m,k_m}}{\sqrt{N} q} + \OO_{d,\bb{p},\eta}\Big(\big(\frac{\|\bb{\delta}_{\bb{k}}\|_1}{\sqrt{N}}\big)^5\Big)
            \end{array}
            \hspace{-1mm}\right\} \notag \\[1mm]
        &= -\sum_{i=1}^d (k_i - N p_i) + \sum_{i\hspace{-0.1mm},\hspace{0.1mm} j = 1}^d \frac{1}{2} \delta_{i,k_i} \delta_{j,k_j} \left\{\hspace{-1mm}
                \begin{array}{l}
                    \frac{1}{q} + \sum_{\ell=1}^d \frac{\delta_{\ell,k_\ell}}{3\sqrt{N} q^2} \\
                    + \sum_{\ell \hspace{-0.1mm},\hspace{-0.1mm} m = 1}^d \frac{\delta_{\ell,k_\ell} \delta_{m,k_m}}{6N q^3} + \OO_{d,\bb{p},\eta}\Big(\big(\frac{\|\bb{\delta}_{\bb{k}}\|_1}{\sqrt{N}}\big)^3\Big)
                \end{array}
                \hspace{-1mm}\right\}.
    \end{align}
    Now, putting \eqref{eq:log.p.eq.1} and \eqref{eq:log.p.eq.2} back into \eqref{eq:log.p}, and using the conditions \eqref{eq:thm:p.k.expansion.condition}, we find
    \begin{equation}
        \begin{aligned}
            \log p_N(\bb{k})
            &= -\log \sqrt{(2\pi N)^d \, p_1 p_2 \dots p_d \hspace{0.3mm} q} \, - \sum_{i\hspace{-0.1mm},\hspace{0.1mm} j = 1}^d \frac{1}{2} \delta_{i,k_i} \delta_{j,k_j} \Big\{\Sigma_{ij}^{-1} + S_{N,ij}\Big\} \\[-1mm]
            &\quad- \frac{1}{2} \sum_{i=1}^d \log\Big(1 + \frac{\delta_{i,k_i}}{\sqrt{N} p_i}\Big) - \frac{1}{2} \log\Big(1 - \sum_{i=1}^d \frac{\delta_{i,k_i}}{\sqrt{N} q}\Big) \\
            &\quad+ \frac{1}{12N} \big\{1 - \sum_{i=1}^d p_i^{-1} - q^{-1}\big\} + \OO_{d,\bb{p},\eta}\bigg(\frac{1 + \|\bb{\delta}_{\bb{k}}\|_1}{N^{3/2}}\bigg),
        \end{aligned}
    \end{equation}
    where the $d \times d$ matrices $\Sigma^{-1}$ and $S_N$ have the $(i,j)$ components:
    \begin{equation}\label{eq:def.M.and.S}
        \begin{aligned}
            \Sigma_{ij}^{-1}
            &\leqdef \frac{1}{p_i} \bb{1}_{\{i = j\}} + \frac{1}{q}, \\
            S_{N,ij}
            &\leqdef \sum_{\ell=1}^d \frac{\delta_{\ell,k_\ell}}{3\sqrt{N}} \bigg\{\frac{-1}{p_i^2} \bb{1}_{\{i = j = \ell\}} + \frac{1}{q^2}\bigg\} \\
            &\quad+ \sum_{\ell \hspace{-0.1mm},\hspace{-0.1mm} m = 1}^d \frac{\delta_{\ell,k_\ell} \delta_{m,k_m}}{6N} \bigg\{\frac{1}{p_i^3} \bb{1}_{\{i = j = \ell = m\}} + \frac{1}{q^3}\bigg\} + \OO_{d,\bb{p},\eta}\bigg(\frac{\|\bb{\delta}_{\bb{k}}\|_1^3}{N^{3/2}}\bigg).
        \end{aligned}
    \end{equation}
    Hence,
    \begin{equation}\label{eq:p.product.exponential}
        \begin{aligned}
            p_N(\bb{k})
            &= \frac{\exp\Big(-\frac{1}{2} \bb{\delta}_{\bb{k}}^{\top} (\Sigma^{-1} + S_N) \, \bb{\delta}_{\bb{k}}\Big)}{\sqrt{(2\pi N)^d \, p_1 p_2 \dots p_d \hspace{0.3mm} q}} \cdot \frac{\prod_{i=1}^d \Big(1 + \frac{\delta_{i,k_i}}{\sqrt{N} p_i}\Big)^{-1/2}}{\Big(1 - \sum_{i=1}^d \frac{\delta_{i,k_i}}{\sqrt{N} q}\Big)^{1/2}} \\
            &\quad\cdot \bigg\{1 + \frac{1}{12N} \big\{1 - \sum_{i=1}^d p_i^{-1} - q^{-1}\big\} + \OO_{d,\bb{p},\eta}\bigg(\frac{1 + \|\bb{\delta}_{\bb{k}}\|_1}{N^{3/2}}\bigg)\bigg\}.
        \end{aligned}
    \end{equation}
    Using the following Taylor expansions, valid for $|y| \leq B < \infty$ and $|x| < \widetilde{\eta} < 1$,
    \vspace{-1mm}
    \begin{equation}
        \begin{aligned}
            e^{-y} &= 1 - y + \frac{y^2}{2} + \OO(e^B y^3), \\[1mm]
            (1 + x)^{-1/2} &= 1 - \frac{x}{2} + \frac{3x^2}{8} + \OO((1 - \widetilde{\eta})^{-7/2} x^3), \\
            (1 - \sum_{i=1}^d x_i)^{-1/2} &= 1 + \sum_{i=1}^d \frac{x_i}{2} + \sum_{i \hspace{-0.1mm},\hspace{-0.1mm} j = 1}^d \frac{3 x_i x_j}{8} + \OO_d((1 - \widetilde{\eta})^{-7/2} \|\bb{x}\|_1^3),
        \end{aligned}
    \end{equation}
    in \eqref{eq:p.product.exponential}, and the function $\phi_{\Sigma}$ from \eqref{eq:phi.M}, we find that
    \begin{equation}\label{eq:p.product.exponential.2}
        \begin{aligned}
            p_N(\bb{k})
            &= N^{-d/2} \phi_{\Sigma}(\bb{\delta}_{\bb{k}}) \cdot \bigg\{1 - \tfrac{1}{2} \bb{\delta}_{\bb{k}}^{\top} S_N \, \bb{\delta}_{\bb{k}} + \tfrac{1}{8} (\bb{\delta}_{\bb{k}}^{\top} S_N \, \bb{\delta}_{\bb{k}})^2 + \OO_{d,\bb{p},\eta}\bigg(\frac{\|\bb{\delta}_{\bb{k}}\|_1^9}{N^{3/2}}\bigg)\bigg\} \\[1mm]
            &\quad \cdot \left\{\hspace{-1mm}
            \begin{array}{l}
                1 - \sum_{i=1}^d \frac{\delta_{i,k_i}}{2\sqrt{N}} \big\{\frac{1}{p_i} - \frac{1}{q}\big\} \\[2mm]
                + \sum_{i \hspace{-0.1mm},\hspace{-0.1mm} j = 1}^d \frac{\delta_{i,k_i} \delta_{j,k_j}}{8N} \big\{\frac{3}{p_i^2} \bb{1}_{\{i = j\}} + \frac{2}{p_i p_j} \bb{1}_{\{i < j\}} - \frac{2}{p_i q} + \frac{3}{q^2}\big\} + \OO_{d,\bb{p},\eta}\Big(\frac{\|\bb{\delta}_{\bb{k}}\|_1^3}{N^{3/2}}\Big)
            \end{array}
            \hspace{-1mm}\right\} \\
            &\quad\cdot \bigg\{1 + \frac{1}{12N} \big\{1 - \sum_{i=1}^d p_i^{-1} - q^{-1}\big\} + \OO_{d,\bb{p},\eta}\bigg(\frac{1 + \|\bb{\delta}_{\bb{k}}\|_1}{N^{3/2}}\bigg)\bigg\}.
        \end{aligned}
    \end{equation}
    By expanding the product of the braces, we get \eqref{eq:LLT.order.2}.
\end{proof}

Before proving Corollary~\ref{cor:approx.proba}, we show that the sum of all $p_N(\bb{k})$'s for which $\bb{k}$ is outside the bulk is negligible.
This is just a specific example of the more general concentration of measure phenomenon, see e.g.\ \cite{MR1849347}.

\begin{lemma}\label{lem:outside.bulk}
    Pick any $\eta\in (0,1)$ and recall the bulk $B_{N\hspace{-0.5mm},\bb{p}}(\eta)$ from \eqref{eq:thm:p.k.expansion.condition}.
    Then,
    \begin{equation}\label{eq:lem:outside.bulk}
        \sum_{\substack{\bb{k}\in \N_0^d\backslash B_{N\hspace{-0.5mm},\bb{p}}(\eta) \\ \|\bb{k}\|_1 \leq N}} \hspace{-1mm}p_N(\bb{k}) = \OO(e^{-\alpha N^{1/3}}), \qquad \text{as $N\to \infty$,}
    \end{equation}
    for some small enough constant $\alpha = \alpha(d,\bb{p},\eta) > 0$.
\end{lemma}

\begin{proof}[Proof of Lemma~\ref{lem:outside.bulk}]
    Notice that if $\bb{k}\in \N_0^d \backslash B_{N\hspace{-0.5mm},\bb{p}}(\eta)$, then at least one component $k_i$, or $N - \|\bb{k}\|_1$, deviates significantly from $N p_i$, or $N q$.
    Therefore, if $\xi_i\sim \text{Binomial}\hspace{0.2mm}(N,p_i)$ for all $i\in \{1,2,\dots,d+1\}$, where $p_{d+1} \leqdef q$, then a union bound followed by Azuma's inequality \cite[Theorem 1.3.1]{MR1422018} yields
    \vspace{-2mm}
    \begin{equation}
        \sum_{\substack{\bb{k}\in \N_0^d\backslash B_{N\hspace{-0.5mm},\bb{p}}(\eta) \\ \|\bb{k}\|_1 \leq N}} \hspace{-1mm}p_N(\bb{k})
        \leq \sum_{i=1}^{d+1} \PP(|\xi_i - N p_i| > \eta \, p_i \hspace{0.3mm} N^{2/3})
        \leq \sum_{i=1}^{d+1} 2 \, e^{- \eta^2 p_i^2 N^{1/3} / 2}.
    \end{equation}
    This ends the proof.
\end{proof}

\begin{proof}[Proof of Corollary~\ref{cor:approx.proba}]
    For any $\bb{y}_0\in \R^d$, we have the Taylor expansion
    \begin{equation}\label{eq:Taylor.phi.Sigma}
        \begin{aligned}
            \phi_{\Sigma}(\bb{y})
            &= \phi_{\Sigma}(\bb{y}_0) + \phi_{\Sigma}'(\bb{y}_0)^{\top} (\bb{y} - \bb{y}_0) + \tfrac{1}{2} (\bb{y} - \bb{y}_0)^{\top} \phi_{\Sigma}''(\bb{y}_0) (\bb{y} - \bb{y}_0) \\[1mm]
            &\qquad+ \OO_{d,\bb{p}}(\phi_{\Sigma}(\bb{y}_0) \|\bb{y} - \bb{y}_0\|_1^3).
        \end{aligned}
    \end{equation}
    If we take $\bb{y}_0 = \bb{\delta}_{\bb{k}}$ and integrate on $\mathcal{H}_{\bb{k}} = [\delta_{1,\scriptscriptstyle k_1 - \frac{1}{2}},\delta_{1,\scriptscriptstyle k_1 + \frac{1}{2}}] \times \dots \times [\delta_{d,\scriptscriptstyle k_d - \frac{1}{2}},\delta_{d,\scriptscriptstyle k_d + \frac{1}{2}}]$, the first and third order derivatives and the second order mixed derivatives ($i \neq j$) disappear because of the symmetry.
    We obtain
    \begin{align}\label{eq:Cressie.generalization.eq.1}
        \int_{\mathcal{H}_{\bb{k}}} \hspace{-1mm} \phi_{\Sigma}(\bb{y}) \rd \bb{y}
        &= N^{-d/2} \phi_{\Sigma}(\bb{\delta}_{\bb{k}}) \cdot \left\{\hspace{-1mm}
            \begin{array}{l}
                1 + \frac{1}{24N} \sum_{i=1}^d \big\{\big([\Sigma^{-1} \bb{\delta}_{\bb{k}}]_i\big)^2 - [\Sigma^{-1}]_{ii}\big\} \\[1.5mm]
                + \OO_{d,\bb{p}}\Big(\frac{(1 + \|\bb{\delta}_{\bb{k}}\|_1)^4}{N^2}\Big)
            \end{array}
            \hspace{-1mm}\right\}.
    \end{align}
    Therefore, for any fixed $\eta\in (0,1)$, say $\eta = 1/2$,
    \begin{equation}
        \begin{aligned}
            &\sum_{\substack{\bb{k}\in A \\ \|\bb{k}\|_1 \leq N}} \hspace{-1mm} p_N(\bb{k}) - \int_{\mathcal{H}_A} \hspace{-1mm} \phi_{\Sigma}(\bb{y}) \rd \bb{y} \\
            &\qquad= \sum_{\substack{\bb{k}\in A \backslash B_{N\hspace{-0.5mm},\bb{p}}(\eta) \\ \|\bb{k}\|_1 \leq N}} \hspace{-1mm} \big(p_N(\bb{k}) - N^{-d/2} \phi_{\Sigma}(\bb{\delta}_{\bb{k}})\big) + \sum_{\substack{\bb{k}\in A \cap B_{N\hspace{-0.5mm},\bb{p}}(\eta) \\ \|\bb{k}\|_1 \leq N}} \hspace{-1mm} \big(p_N(\bb{k}) - N^{-d/2} \phi_{\Sigma}(\bb{\delta}_{\bb{k}})\big) \\
            &\qquad\qquad- \frac{1}{24N} \sum_{i=1}^d \sum_{\substack{\bb{k}\in A \\ \|\bb{k}\|_1 \leq N}} \hspace{-1mm} \big\{\big([\Sigma^{-1} \bb{\delta}_{\bb{k}}]_i\big)^2 - [\Sigma^{-1}]_{ii}\big\} \, N^{-d/2} \phi_{\Sigma}(\bb{\delta}_{\bb{k}}) + \OO_{d,\bb{p}}(N^{-2}).
        \end{aligned}
    \end{equation}
    The first sum on the right-hand side is exponentially small in $N^{1/3}$ by Lemma~\ref{lem:outside.bulk} (and an analogous estimate for the multivariate normal distribution), and the terms in the second sum are estimated using Theorem~\ref{thm:p.k.expansion}.
    The conclusion follows.
\end{proof}

\begin{appendices}
\section{Technical lemmas}

Below are the joint central moments (up to three) of the multinomial distribution.
These moments were obtained in \cite{arXiv:2006.09059} by differentiating the moment generating function, cf.\ \cite{Ouimet_2021_multinomial_moments}.
This lemma is used to estimate the $\asymp N^{-1}$ errors in \eqref{eq:estimate.I.begin} of the proof of Lemma~\ref{lem:prelim.Carter}, and also as a preliminary result for the proof of Lemma~\ref{lem:Leblanc.2012.boundary.Lemma.1.with.set.A} below.

\begin{lemma}[Joint central moments $1$ to $3$]\label{lem:Leblanc.2012.boundary.Lemma.1}
    Let $\bb{p}\in (0,1)^d$ be such that $\|\bb{p}\|_1 < 1$.
    If $\bb{\xi} = (\xi_1,\xi_2,\dots,\xi_d)\sim \mathrm{Multinomial}\hspace{0.2mm}(N,\bb{p})$ according to \eqref{eq:multinomial.pdf}, then, for all $i,j,\ell\in \{1,2,\dots,d\}$,
    \begin{align}
        &\EE\big[\xi_i - N p_i\big] = 0, \label{eq:thm:central.moments.eq.1} \\[2mm]
        &\EE\big[(\xi_i - N p_i)(\xi_j - N p_j)\big] = N \, (p_i \ind_{\{i = j\}} - p_i p_j), \label{eq:thm:central.moments.eq.2} \\[2mm]
        &\EE\big[(\xi_i - N p_i)(\xi_j - N p_j)(\xi_{\ell} - N p_{\ell})\big] \label{eq:thm:central.moments.eq.3} \\[1mm]
        &\hspace{3mm}= N \big(2 p_i p_j p_{\ell} - \ind_{\{i = j\}} p_i p_{\ell} - \ind_{\{j = \ell\}} p_i p_j - \ind_{\{i = \ell\}} p_j p_{\ell} + \ind_{\{i = j = \ell\}} p_i\big). \notag
    \end{align}
\end{lemma}

We can also estimate the moments of Lemma~\ref{lem:Leblanc.2012.boundary.Lemma.1} on various events.
The lemma below is used to estimate the $\asymp N^{-1/2}$ errors in \eqref{eq:estimate.I.begin} of the proof of Lemma~\ref{lem:prelim.Carter}.

\begin{lemma}\label{lem:Leblanc.2012.boundary.Lemma.1.with.set.A}
    Let $\bb{p}\in (0,1)^d$ be such that $\|\bb{p}\|_1 < 1$, and let $A\in \mathscr{B}(\R^d)$ be a Borel set.
    If $\bb{\xi} = (\xi_1,\xi_2,\dots,\xi_d)\sim \mathrm{Multinomial}\hspace{0.2mm}(N,\bb{p})$ according to \eqref{eq:multinomial.pdf}, then, for all $i,j,\ell\in \{1,2,\dots,d\}$,
    \begin{align}
        &\Big|\EE\big[(\xi_i - N p_i) \, \ind_{\{\bb{\xi}\in A\}}\big]\Big| \leq \frac{1}{2} N^{1/2} \big(\PP(\bb{\xi}\in A^c)\big)^{1/2}, \label{eq:thm:central.moments.eq.1.set.A} \\[2mm]
        &\Big|\EE\big[(\xi_i - N p_i)(\xi_j - N p_j) \, \ind_{\{\bb{\xi}\in A\}}\big] - N \, (p_i \ind_{\{i = j\}} - p_i p_j)\Big| \leq \frac{1}{2} N \big(\PP(\bb{\xi}\in A^c)\big)^{1/2}, \label{eq:thm:central.moments.eq.2.set.A} \\[2mm]
        &\left|\hspace{-1mm}
            \begin{array}{l}
                \EE\big[(\xi_i - N p_i)(\xi_j - N p_j)(\xi_{\ell} - N p_{\ell}) \, \ind_{\{\bb{\xi}\in A\}}\big] \\[1mm]
                 - N \bigg(\hspace{-1mm}
                    \begin{array}{l}
                        2 p_i p_j p_{\ell} - \ind_{\{i = j\}} p_i p_{\ell} - \ind_{\{j = \ell\}} p_i p_j \\
                        - \ind_{\{i = \ell\}} p_j p_{\ell} + \ind_{\{i = j = \ell\}} p_i
                    \end{array}
                    \hspace{-1mm}\bigg)
            \end{array}
            \hspace{-1mm}\right| \leq \frac{1}{\sqrt{8}} N^{3/2} \big(\PP(\bb{\xi}\in A^c)\big)^{1/4} \label{eq:thm:central.moments.eq.3.set.A}.
    \end{align}
\end{lemma}

\begin{proof}
    For the bound in \eqref{eq:thm:central.moments.eq.1.set.A}, Equation \eqref{eq:thm:central.moments.eq.1}, Cauchy-Schwarz and a standard bound on the variance of the binomial distribution yield
    \begin{align}\label{eq:thm:central.moments.eq.1.set.A.proof}
        \Big|\EE\big[(\xi_i - N p_i) \, \ind_{\{\bb{\xi}\in A\}}\big]\Big|
        &= \Big|\EE\big[(\xi_i - N p_i) \, \ind_{\{\bb{\xi}\in A^c\}}\big]\Big| \notag \\
        &\leq \big(\EE\big[(\xi_i - N p_i)^2\big]\big)^{1/2} \big(\PP(\bb{\xi}\in A^c)\big)^{1/2} \notag \\
        &\leq \frac{1}{2} N^{1/2} \big(\PP(\bb{\xi}\in A^c)\big)^{1/2}.
    \end{align}
    For the bound in \eqref{eq:thm:central.moments.eq.2.set.A}, Equation \eqref{eq:thm:central.moments.eq.2}, Holder's inequality and a standard bound on the fourth central moment of the binomial distribution yield
    \begin{align}\label{eq:thm:central.moments.eq.2.set.A.proof}
        &\Big|\EE\big[(\xi_i - N p_i)(\xi_j - N p_j) \, \ind_{\{\bb{\xi}\in A\}}\big] - N \, (p_i \ind_{\{i = j\}} - p_i p_j)\Big| \notag \\
        &\quad= \Big|\EE\big[(\xi_i - N p_i)(\xi_j - N p_j) \, \ind_{\{\bb{\xi}\in A^c\}}\big]\Big| \notag \\
        &\quad\leq \big(\EE\big[(\xi_i - N p_i)^4\big]\big)^{1/4} \big(\EE\big[(\xi_j - N p_j)^4\big]\big)^{1/4} \big(\PP(\bb{\xi}\in A^c)\big)^{1/2} \notag \\[1mm]
        &\quad\leq \big(\tfrac{1}{4} N^2\big)^{1/4} \big(\tfrac{1}{4} N^2\big)^{1/4} \big(\PP(\bb{\xi}\in A^c)\big)^{1/2} \notag \\[1mm]
        &\quad= \frac{1}{2} N \big(\PP(\bb{\xi}\in A^c)\big)^{1/2}.
    \end{align}
    For the bound in \eqref{eq:thm:central.moments.eq.3.set.A}, Equation \eqref{eq:thm:central.moments.eq.3}, Holder's inequality and a standard bound on the fourth central moment of the binomial distribution yield
    \begin{align}\label{eq:thm:central.moments.eq.3.set.A.proof}
        &\left|\hspace{-1mm}
            \begin{array}{l}
                \EE\big[(\xi_i - N p_i)(\xi_j - N p_j)(\xi_{\ell} - N p_{\ell}) \, \ind_{\{\bb{\xi}\in A\}}\big] \\[1mm]
                 - N \, \big(2 p_i p_j p_{\ell} - \ind_{\{i = j\}} p_i p_{\ell} - \ind_{\{j = \ell\}} p_i p_j - \ind_{\{i = \ell\}} p_j p_{\ell} + \ind_{\{i = j = \ell\}} p_i\big)
            \end{array}
            \hspace{-1mm}\right| \notag \\[0.5mm]
        &\quad= \Big|\EE\big[(\xi_i - N p_i)(\xi_j - N p_j)(\xi_{\ell} - N p_{\ell}) \, \ind_{\{\bb{\xi}\in A^c\}}\big]\Big| \notag \\[0.5mm]
        &\quad\leq \big(\EE\big[(\xi_i - N p_i)^4\big]\big)^{1/4} \big(\EE\big[(\xi_j - N p_j)^4\big]\big)^{1/4} \big(\EE\big[(\xi_{\ell} - N p_{\ell})^4\big]\big)^{1/4} \big(\PP(\bb{\xi}\in A^c)\big)^{1/4} \notag \\[1.5mm]
        &\quad\leq \big(\tfrac{1}{4} N^2\big)^{1/4} \big(\tfrac{1}{4} N^2\big)^{1/4} \big(\tfrac{1}{4} N^2\big)^{1/4} \big(\PP(\bb{\xi}\in A^c)\big)^{1/4} \notag \\[0.5mm]
        &\quad= \frac{1}{\sqrt{8}} N^{3/2} \big(\PP(\bb{\xi}\in A^c)\big)^{1/4}.
    \end{align}
    This ends the proof.
\end{proof}

For the joint central moments $4$ and $6$, we have the following results.
These estimates are crucial to bound the $\asymp N^{-1}$ errors in \eqref{eq:estimate.I.begin} of the proof of Lemma~\ref{lem:prelim.Carter}.

\begin{lemma}[Joint central moments $4$ and $6$]\label{lem:moments.4.to.6}
    Let $\bb{p}\in (0,1)^d$ be such that $\|\bb{p}\|_1 < 1$.
    If $\bb{\xi} = (\xi_1,\xi_2,\dots,\xi_d)\sim \mathrm{Multinomial}\hspace{0.2mm}(N,\bb{p})$ according to \eqref{eq:multinomial.pdf}, then, for all $i \neq j$ in $\{1,2,\dots,d\}$,
    \begin{align}
        &\EE\big[(\xi_i - N p_i)^4\big] = 3 \, N^2 p_i^2 \, (1 - p_i)^2 + \OO(N), \label{eq:thm:central.moments.4.0} \\[2mm]
        &\EE\big[(\xi_i - N p_i)^6\big] = 15 \, N^3 p_i^3 \, (1 - p_i)^3 + \OO(N^2), \label{eq:thm:central.moments.6.0} \\[2mm]
        &\EE\big[(\xi_i - N p_i)^3 (\xi_j - N p_j)^3\big] = N^3 p_i^2 p_j^2 \, \big(-9 + 9 \, p_j + 9 \, p_i - 15 \, p_i p_j\big) + \OO(N^2). \label{eq:thm:central.moments.3.3}
    \end{align}
\end{lemma}

\begin{proof}
    If we denote $x^{(r)} \leqdef x (x - 1) (x - 2) \dots (x - r + 1)$, we know from \cite[p.67]{MR143299} that, for all $\bb{a} \leqdef (a_1,a_2,\dots,a_d)\in \N_0^d$,
    \begin{equation}
        \EE\big[\xi_1^{(a_1)} \xi_2^{(a_2)} \dots \xi_d^{(a_d)}\big] = N^{(\|\bb{a}\|_1)} p_1^{a_1} p_2^{a_2} \dots p_d^{a_d}.
    \end{equation}
    Hence, for all $i \neq j$ in $\{1,2,\dots,d\}$,
    \begin{align}
        \EE\big[\xi_i\big]
        &= N p_i, \\[1.5mm]
        \EE\big[\xi_i^2\big]
        &= \EE\big[\xi_i + \xi_i^{(2)}\big] \notag \\
        &= N p_i + N^{(2)} p_i^2, \\[1.5mm]
        \EE\big[\xi_i \xi_j\big]
        &= N^{(2)} p_i p_j, \\[1.5mm]
        \EE\big[\xi_i^3\big]
        &= \EE\big[\xi_i + 3 \, \xi_i^{(2)} + \xi_i^{(3)}\big] \notag \\
        &= N p_i + 3 \, N^{(2)} p_i^2 + N^{(3)} p_i^3, \\[1.5mm]
        \EE\big[\xi_i^2 \xi_j\big]
        &= \EE\big[(\xi_i + \xi_i^{(2)}) \xi_j\big] \notag \\
        &= N^{(2)} p_i p_j + N^{(3)} p_i^2 p_j, \\[1.5mm]
        \EE\big[\xi_i \xi_j^2\big]
        &= N^{(2)} p_i p_j + N^{(3)} p_i p_j^2, \\[1.5mm]
        \EE\big[\xi_i^4\big]
        &= \EE\big[\xi_i + 7 \, \xi_i^{(2)} + 6 \, \xi_i^{(3)} + \xi_i^{(4)}\big] \notag \\
        &= N p_i + 7 \, N^{(2)} p_i^2 + 6 \, N^{(3)} p_i^3 + N^{(4)} p_i^4, \\[1.5mm]
        \EE\big[\xi_i^3 \, \xi_j\big]
        &= \EE\big[(\xi_i + 3 \, \xi_i^{(2)} + \xi_i^{(3)}) \xi_j\big] \notag \\
        &= \EE\big[\xi_i \xi_j\big] + 3 \, \EE\big[\xi_i^{(2)} \xi_j\big] + \EE\big[\xi_i^{(3)} \xi_j\big] \notag \\
        &= N^{(2)} p_i p_j + 3 \, N^{(3)} p_i^2 p_j + N^{(4)} p_i^3 p_j, \\[1.5mm]
        \EE\big[\xi_i \xi_j^3\big]
        &= N^{(2)} p_i p_j + 3 \, N^{(3)} p_i p_j^2 + N^{(4)} p_i p_j^3, \\[1.5mm]
        \EE\big[\xi_i^2 \xi_j^2\big]
        &= \EE\big[(\xi_i + \xi_i^{(2)}) (\xi_j + \xi_j^{(2)})\big] \notag \\
        &= \EE\big[\xi_i \xi_j\big] + \EE\big[\xi_i \xi_j^{(2)}\big] + \EE\big[\xi_i^{(2)} \xi_j\big] + \EE\big[\xi_i^{(2)} \xi_j^{(2)}\big] \notag \\
        &= N^{(2)} p_i p_j + N^{(3)} p_i p_j^2 + N^{(3)} p_i^2 p_j + N^{(4)} p_i^2 p_j^2, \\[1.5mm]
        \EE\big[\xi_i^5\big]
        &= \EE\big[\xi_i + 15 \, \xi_i^{(2)} + 25 \, \xi_i^{(3)} + 10 \, \xi_i^{(4)} + \xi_i^{(5)}\big] \notag \\
        &= N p_i + 15 \, N^{(2)} p_i^2 + 25 \, N^{(3)} p_i^3 + 10 \, N^{(4)} p_i^4 + N^{(5)} p_i^5, \\[1.5mm]
        \EE\big[\xi_i^3 \, \xi_j^2\big]
        &= \EE\big[(\xi_i + 3 \, \xi_i^{(2)} + \xi_i^{(3)}) (\xi_j + \xi_j^{(2)})\big] \notag \\
        &= \EE\big[\xi_i \xi_j\big] + \EE\big[\xi_i \xi_j^{(2)}\big] + 3 \, \EE\big[\xi_i^{(2)} \xi_j\big] \notag \\
        &\quad+ 3 \, \EE\big[\xi_i^{(2)} \xi_j^{(2)}\big] + \EE\big[\xi_i^{(3)} \xi_j\big] + \EE\big[\xi_i^{(3)} \xi_j^{(2)}\big] \notag \\
        &= N^{(2)} p_i p_j + N^{(3)} p_i p_j^2 + 3 \, N^{(3)} p_i^2 p_j + 3 \, N^{(4)} p_i^2 p_j^2 + N^{(4)} p_i^3 p_j + N^{(5)} p_i^3 p_j^2, \\[1.5mm]
        \EE\big[\xi_i^2 \xi_j^3\big]
        &= N^{(2)} p_i p_j + N^{(3)} p_i^2 p_j + 3 \, N^{(3)} p_i p_j^2 + 3 \, N^{(4)} p_i^2 p_j^2 + N^{(4)} p_i p_j^3 + N^{(5)} p_i^2 p_j^3, \\[1.5mm]
        \EE\big[\xi_i^6\big]
        &= \EE\big[\xi_i + 31 \, \xi_i^{(2)} + 90 \, \xi_i^{(3)} + 65 \, \xi_i^{(4)} + 15 \, \xi_i^{(5)} + \xi_i^{(6)}\big] \notag \\
        &= N p_i + 31 \, N^{(2)} p_i^2 + 90 \, N^{(3)} p_i^3 + 65 \, N^{(4)} p_i^4 + 15 \, N^{(5)} p_i^5 + N^{(6)} p_i^6, \\[1.5mm]
        \EE\big[\xi_i^3 \, \xi_j^3\big]
        &= \EE\big[(\xi_i + 3 \, \xi_i^{(2)} + \xi_i^{(3)}) (\xi_j + 3 \, \xi_j^{(2)} + \xi_j^{(3)})\big] \notag \\
        &= \EE\big[\xi_i \xi_j\big] + 3 \, \EE\big[\xi_i \xi_j^{(2)}\big] + \EE\big[\xi_i \xi_j^{(3)}\big] + 3 \, \EE\big[\xi_i^{(2)} \xi_j\big] + 9 \, \EE\big[\xi_i^{(2)} \xi_j^{(2)}\big] \notag \\
        &\quad+ 3 \, \EE\big[\xi_i^{(2)} \xi_j^{(3)}\big] + \EE\big[\xi_i^{(3)} \xi_j\big] + 3 \, \EE\big[\xi_i^{(3)} \xi_j^{(2)}\big] + \EE\big[\xi_i^{(3)} \xi_j^{(3)}\big] \notag \\
        &= N^{(2)} p_i p_j + 3 \, N^{(3)} p_i p_j^2 + N^{(4)} p_i p_j^3 + 3 \, N^{(3)} p_i^2 p_j + 9 \, N^{(4)} p_i^2 p_j^2 \notag \\
        &\quad+ 3 \, N^{(5)} p_i^2 p_j^3 + N^{(4)} p_i^3 p_j + 3 \, N^{(5)} p_i^3 p_j^2 + N^{(6)} p_i^3 p_j^3.
    \end{align}
    We deduce
    \begin{align}\label{eq:calculation.power.4}
        \EE\big[(\xi_i - N p_i)^4\big]
        &= \EE\big[\xi_i^4\big] - 4 \, (N p_i) \, \EE\big[\xi_i^3\big] + 6 \, (N p_i)^2 \, \EE\big[\xi_i^2\big] - 4 \, (N p_i)^3 \, \EE\big[\xi_i\big] + (N p_i)^4 \notag \\
        &= \big(N p_i + 7 \, N^{(2)} p_i^2 + 6 \, N^{(3)} p_i^3 + N^{(4)} p_i^4\big) \notag \\
        &\quad- 4 \, (N p_i) \, \big(N p_i + 3 \, N^{(2)} p_i^2 + N^{(3)} p_i^3\big) \notag \\
        &\quad+ 6 \, (N p_i)^2 \, \big(N p_i + N^{(2)} p_i^2\big) - 4 \, (N p_i)^3 \, \big(N p_i\big) + (N p_i)^4 \notag \\
        &= \bigg(\hspace{-1mm}
            \begin{array}{l}
                N p_i + 7 \, N^2 p_i^2 - 7 \, N p_i^2 + 6 \, N^3 p_i^3 - 18 \, N^2 p_i^3 \\
                + 12 \, N p_i^3 + N^4 p_i^4 - 6 \, N^3 p_i^4 + 11 \, N^2 p_i^4 - 6 \, N p_i^4
            \end{array}
            \hspace{-1mm}\bigg) \notag \\
        &\quad- 4 \, (N p_i) \, \big(N p_i + 3 \, N^2 p_i^2 - 3 \, N p_i^2 + N^3 p_i^3 - 3 \, N^2 p_i^3 + 2 \, N p_i^3\big) \notag \\
        &\quad+ 6 \, (N p_i)^2 \, \big(N p_i + N^2 p_i^2 - N p_i^2\big) - 4 \, (N p_i)^3 \, \big(N p_i\big) + (N p_i)^4 \notag \\
        &= N p_i + 7 \, N^2 p_i^2 - 7 \, N p_i^2 - 18 \, N^2 p_i^3 + 12 \, N p_i^3 - 6 \, N^3 p_i^4 + 11 \, N^2 p_i^4 \notag \\
        &\quad- 6 \, N p_i^4 - 4 \, N^2 p_i^2 + 12 \, N^2 p_i^3 + 12 \, N^3 p_i^4 - 8 \, N^2 p_i^4 - 6 \, N^3 p_i^4 \notag \\
        &= N p_i + 3 \, N^2 p_i^2 - 7 \, N p_i^2 - 6 \, N^2 p_i^3 + 12 \, N p_i^3 + 3 \, N^2 p_i^4 - 6 \, N p_i^4 \notag \\
        &= N^2 p_i^2 \, \big(3 - 6 \, p_i + 3 \, p_i^2\big) + \OO(N),
    \end{align}
    which proves the claim in \eqref{eq:thm:central.moments.4.0}.
    In a similar manner,
    \begin{align}\label{eq:calculation.power.6}
        \EE\big[(\xi_i - N p_i)^6\big]
        &= \EE\big[\xi_i^6\big] - 6 \, (N p_i) \, \EE\big[\xi_i^5\big] + 15 \, (N p_i)^2 \, \EE\big[\xi_i^4\big] - 20 \, (N p_i)^3 \, \EE\big[\xi_i^3\big] \notag \\
        &\quad+ 15 \, (N p_i)^4 \, \EE\big[\xi_i^2\big] - 6 \, (N p_i)^5 \, \EE\big[\xi_i\big] + (N p_i)^6 \notag \\
        &= \big(N p_i + 31 \, N^{(2)} p_i^2 + 90 \, N^{(3)} p_i^3 + 65 \, N^{(4)} p_i^4 + 15 \, N^{(5)} p_i^5 + N^{(6)} p_i^6\big) \notag \\
        &\quad- 6 \, (N p_i) \big(N p_i + 15 \, N^{(2)} p_i^2 + 25 \, N^{(3)} p_i^3 + 10 \, N^{(4)} p_i^4 + N^{(5)} p_i^5\big) \notag \\
        &\quad+ 15 \, (N p_i)^2 \big(N p_i + 7 \, N^{(2)} p_i^2 + 6 \, N^{(3)} p_i^3 + N^{(4)} p_i^4\big) \notag \\
        &\quad- 20 \, (N p_i)^3 \big(N p_i + 3 \, N^{(2)} p_i^2 + N^{(3)} p_i^3\big) \notag \\
        &\quad+ 15 \, (N p_i)^4 \big(N p_i + N^{(2)} p_i^2\big) - 6 \, (N p_i)^5 \big(N p_i\big) + (N p_i)^6.
    \end{align}
    It is well known that $\EE\big[(\xi_i - N p_i)^6\big] = \OO(N^3)$ for the binomial, so the terms with powers $N^4$ and above must cancel out in \eqref{eq:calculation.power.6}.
    Therefore, we get
    \begin{align}
        &\EE\big[(\xi_i - N p_i)^6\big] \notag \\
        &\quad= 90 \cdot 1 \, N^3 p_i^3 + 65 \cdot (-6) \, N^3 p_i^4 + 15 \cdot 35 \, N^3 p_i^5 - 225 \, N^3 p_i^6 \notag \\
        &\quad\quad- 6 \cdot 15 \, N^3 p_i^3 - 6 \cdot 25 \cdot (-3) \, N^3 p_i^4 - 6 \cdot 10 \cdot 11 \, N^3 p_i^5 - 6 \cdot (-50) \, N^3 p_i^6 \notag \\
        &\quad\quad+ 15 \, N^3 p_i^3 + 15 \cdot 7 \cdot (-1) \, N^3 p_i^4 + 15 \cdot 6 \cdot 2 \, N^3 p_i^5 + 15 \cdot (-6) \, N^3 p_i^6 + \OO(N^2) \notag \\
        &\quad= N^3 p_i^3 \, \big(15 - 45 \, p_i + 45 \, p_i^2 - 15 \, p_i^3\big) + \OO(N^2),
    \end{align}
    which proves the claim in \eqref{eq:thm:central.moments.6.0}.
    For all $i \neq j$ in $\{1,2,\dots,d\}$, we also have
    \begin{align}
        &\EE\big[(\xi_i - N p_i)^3 (\xi_j - N p_j)^3\big] \notag \\[1mm]
        &\quad= \EE\left[\hspace{-1mm}
            \begin{array}{l}
                (\xi_i^3 - 3 \, \xi_i^2 N p_i + 3 \, \xi_i (N p_i)^2 - (N p_i)^3) \\[1mm]
                \cdot \, (\xi_j^3 - 3 \, \xi_j^2 N p_j + 3 \, \xi_j (N p_j)^2 - (N p_j)^3)
            \end{array}
            \hspace{-1mm}\right] \notag \\[1.5mm]
        &\quad= \EE\big[\xi_i^3 \, \xi_j^3\big] - 3 \, N p_j \, \EE\big[\xi_i^3 \, \xi_j^2\big] + 3 \, N^2 p_j^2 \, \EE\big[\xi_i^3 \, \xi_j\big] - N^3 p_j^3 \, \EE\big[\xi_i^3\big] \notag \\
        &\quad\quad- 3 \, N p_i \, \EE\big[\xi_i^2 \xi_j^3\big] + 9 \, N^2 p_i p_j \, \EE\big[\xi_i^2 \xi_j^2\big] - 9 \, N^3 p_i p_j^2 \, \EE\big[\xi_i^2 \xi_j\big] + 3 \, N^4 p_i p_j^3 \, \EE\big[\xi_i^2\big] \notag \\
        &\quad\quad+ 3 \, N^2 p_i^2 \, \EE\big[\xi_i \xi_j^3\big] - 9 \, N^3 p_i^2 p_j \, \EE\big[\xi_i \xi_j^2\big] + 9 \, N^4 p_i^2 p_j^2 \, \EE\big[\xi_i \xi_j\big] - 3 \, N^5 p_i^2 p_j^3 \, \EE\big[\xi_i\big] \notag \\
        &\quad\quad- N^3 p_i^3 \, \EE\big[\xi_j^3\big] + 3 \, N^4 p_i^3 p_j \, \EE\big[\xi_j^2\big] - 3 \, N^5 p_i^3 p_j^2 \, \EE\big[\xi_j\big] + N^6 p_i^3 p_j^3 \notag \\[1.5mm]
        &\quad= \bigg(\hspace{-1mm}
            \begin{array}{l}
                N^{(2)} p_i p_j + 3 \, N^{(3)} p_i p_j^2 + N^{(4)} p_i p_j^3  + 3 \, N^{(3)} p_i^2 p_j + 9 \, N^{(4)} p_i^2 p_j^2 \\
                + 3 \, N^{(5)} p_i^2 p_j^3  + N^{(4)} p_i^3 p_j + 3 \, N^{(5)} p_i^3 p_j^2 + N^{(6)} p_i^3 p_j^3
            \end{array}
            \hspace{-1mm}\bigg) \notag \\
        &\quad\quad- 3 \, N p_j \bigg(\hspace{-1mm}
            \begin{array}{l}
                N^{(2)} p_i p_j + N^{(3)} p_i p_j^2 + 3 \, N^{(3)} p_i^2 p_j \\
                + 3 \, N^{(4)} p_i^2 p_j^2 + N^{(4)} p_i^3 p_j + N^{(5)} p_i^3 p_j^2
            \end{array}
            \hspace{-1mm}\bigg) \notag \\
        &\quad\quad+ 3 \, N^2 p_j^2 \big(N^{(2)} p_i p_j + 3 \, N^{(3)} p_i^2 p_j + N^{(4)} p_i^3 p_j\big) \notag \\
        &\quad\quad- N^3 p_j^3 \big(N p_i + 3 \, N^{(2)} p_i^2 + N^{(3)} p_i^3\big) \notag \\
        &\quad\quad- 3 \, N p_i \bigg(\hspace{-1mm}
            \begin{array}{l}
                N^{(2)} p_i p_j + N^{(3)} p_i^2 p_j + 3 \, N^{(3)} p_i p_j^2 \\
                + 3 \, N^{(4)} p_i^2 p_j^2 + N^{(4)} p_i p_j^3 + N^{(5)} p_i^2 p_j^3
            \end{array}
            \hspace{-1mm}\bigg) \notag \\
        &\quad\quad+ 9 \, N^2 p_i p_j \big(N^{(2)} p_i p_j + N^{(3)} p_i p_j^2 + N^{(3)} p_i^2 p_j + N^{(4)} p_i^2 p_j^2\big) \notag \\
        &\quad\quad- 9 \, N^3 p_i p_j^2 \big(N^{(2)} p_i p_j + N^{(3)} p_i^2 p_j\big) + 3 \, N^4 p_i p_j^3 \big(N p_i + N^{(2)} p_i^2\big) \notag \\
        &\quad\quad+ 3 \, N^2 p_i^2 \, \big(N^{(2)} p_i p_j + 3 \, N^{(3)} p_i p_j^2 + N^{(4)} p_i p_j^3\big) \notag \\
        &\quad\quad- 9 \, N^3 p_i^2 p_j \big(N^{(2)} p_i p_j + N^{(3)} p_i p_j^2\big) + 9 \, N^4 p_i^2 p_j^2 \big(N^{(2)} p_i p_j\big) - 3 \, N^5 p_i^2 p_j^3 \big(N p_i\big) \notag \\
        &\quad\quad- N^3 p_i^3 \big(N p_j + 3 \, N^{(2)} p_j^2 + N^{(3)} p_j^3\big) + 3 \, N^4 p_i^3 p_j \big(N p_j + N^{(2)} p_j^2\big) \notag \\
        &\quad\quad- 3 \, N^5 p_i^3 p_j^2 \big(N p_j\big) + N^6 p_i^3 p_j^3.
    \end{align}
    All the terms with powers $N^4$ and above must cancel out because otherwise there would be a constant $c = c(p_i,p_j) > 0$ small enough that
    \begin{equation}
        \begin{aligned}
            c_{p_i,p_j} N^4
            &\leq \big|\EE\big[(\xi_i - N p_i)^3 (\xi_j - N p_j)^3\big]\big| \\
            &\leq \sqrt{\EE\big[(\xi_i - N p_i)^6\big]} \sqrt{\EE\big[(\xi_i - N p_i)^6\big]} \\
            &= \OO(\sqrt{N^3} \sqrt{N^3}),
        \end{aligned}
    \end{equation}
    which is a contradiction for $N$ large enough.
    Therefore, we can write
    \begin{align}
        &\EE\big[(\xi_i - N p_i)^3 (\xi_j - N p_j)^3\big] \notag \\[1mm]
        &\quad= 3 \, N^3 p_i p_j^2 - 6 \, N^3 p_i p_j^3 + 3 \, N^3 p_i^2 p_j + 9 \cdot (-6) \, N^3 p_i^2 p_j^2 \notag \\
        &\quad\quad+ 3 \cdot 35 \, N^3 \, p_i^2 p_j^3 - 6 \, N^3 p_i^3 p_j + 3 \cdot 35 N^3 p_i^3 p_j^2 - 225 \, N^3 p_i^3 p_j^3 \notag \\
        &\quad\quad- 3 \cdot 1 \, N^3 p_i p_j^2 - 3 \cdot (-3) \, N^3 p_i p_j^3 - 3 \cdot 3 \cdot (-3) \, N^3 p_i^2 p_j^2 \notag \\
        &\quad\quad- 3 \cdot 3 \cdot 11 \, N^3 p_i^2 p_j^3 - 3 \cdot 11 \, N^3 p_i^3 p_j^2 - 3 \cdot (-50) \, N^3 p_i^3 p_j^3 \notag \\
        &\quad\quad+ 3 \cdot (-1) \, N^3 p_i p_j^3 + 3 \cdot 3 \cdot 2 \, N^3 p_i^2 p_j^3 + 3 \cdot (-6) \, N^3 p_i^3 p_j^3 \notag \\
        &\quad\quad- 3 \cdot 1 \, N^3 p_i^2 p_j - 3 \cdot (-3) \, N^3 p_i^3 p_j - 3 \cdot 3 \cdot (-3) \, N^3 p_i^2 p_j^2 \notag \\
        &\quad\quad- 3 \cdot 3 \cdot 11 \, N^3 p_i^3 p_j^2 - 3 \cdot 11 \, N^3 p_i^2 p_j^3 - 3 \cdot (-50) \, N^3 p_i^3 p_j^3 \notag \\
        &\quad\quad+ 9 \cdot (-1) \, N^3 p_i^2 p_j^2 + 9 \cdot 2 \, N^3 p_i^2 p_j^3 + 9 \cdot 2 \, N^3 p_i^3 p_j^2 + 9 \cdot (-6) \, N^3 p_i^3 p_j^3 \notag \\
        &\quad\quad+ 3 \cdot (-1) \, N^3 p_i^3 p_j + 3 \cdot 3 \cdot 2 \, N^3 p_i^3 p_j^2 + 3 \cdot (-6) \, N^3 p_i^3 p_j^3 + \OO(N^2).
    \end{align}
    Now, the terms with a factor of the form $p_i^2 p_j$ or $p_i p_j^2$ or $p_i^3 p_j$ or $p_i p_j^3$ all cancel out with each other, so we can simplify to
    \begin{align}
        &\EE\big[(\xi_i - N p_i)^3 (\xi_j - N p_j)^3\big] \notag \\[1mm]
        &\quad= -54 \, N^3 p_i^2 p_j^2 + 105 \, N^3 \, p_i^2 p_j^3 + 105 \, N^3 p_i^3 p_j^2 - 225 \, N^3 p_i^3 p_j^3 \notag \\
        &\quad\quad+ 27 \, N^3 p_i^2 p_j^2 - 99 \, N^3 p_i^2 p_j^3 - 33 \, N^3 p_i^3 p_j^2 + 150 \, N^3 p_i^3 p_j^3 \notag \\
        &\quad\quad+ 18 \, N^3 p_i^2 p_j^3 - 18 \, N^3 p_i^3 p_j^3 + 27 \, N^3 p_i^2 p_j^2 - 99 \, N^3 p_i^3 p_j^2 \notag \\
        &\quad\quad- 33 \, N^3 p_i^2 p_j^3 + 150 \, N^3 p_i^3 p_j^3 - 9 \, N^3 p_i^2 p_j^2 + 18 \, N^3 p_i^2 p_j^3 \notag \\
        &\quad\quad+ 18 \, N^3 p_i^3 p_j^2 - 54 \, N^3 p_i^3 p_j^3 + 18 \, N^3 p_i^3 p_j^2 - 18 \, N^3 p_i^3 p_j^3 + \OO(N^2) \notag \\
        &\quad= N^3 p_i^2 p_j^2 \big(-9 + 9 \, p_j + 9 \, p_i - 15 \, p_i p_j\big) + \OO(N^2),
    \end{align}
    which proves the claim in \eqref{eq:thm:central.moments.3.3}.
\end{proof}

\end{appendices}

\section*{Acknowledgments}

We thank the referees for their useful comments, in particular for bringing up the reference \cite{MR750392}.
The author acknowledges support of a postdoctoral fellowship from the NSERC (PDF) and a supplement from the FRQNT (B3X).

%
%

\phantomsection
\addcontentsline{toc}{chapter}{References}

\bibliographystyle{authordate1}
\bibliography{Ouimet_2020_LLT_multinomial_bib}

\end{document}